\documentclass[10pt]{amsart}

\author{Ties Laarakker}
\address{}
\email{p.t.a.laarakker@uu.nl}
\title{Monopole contributions to refined Vafa-Witten invariants}

\usepackage[utf8]{inputenc}
\usepackage{amsmath}
\usepackage{amsfonts}
\usepackage{amssymb}
\usepackage{amsthm}
\usepackage{hyperref}
\usepackage{mathtools}
\usepackage{mathrsfs}
\usepackage{tikz-cd}
\usepackage{url}

\newcommand{\CC}{\mathbb{C}}
\newcommand{\PP}{\mathbb{P}}
\newcommand{\QQ}{\mathbb{Q}}
\newcommand{\ZZ}{\mathbb{Z}}

\newcommand{\C}{\mathcal{C}}
\newcommand{\I}{\mathcal{I}}

\renewcommand{\L}{\mathcal{L}}
\newcommand{\M}{\mathcal{M}}
\renewcommand{\O}{\mathcal{O}}
\newcommand{\N}{\mathcal{N}}
\newcommand{\SW}{\mathrm{SW}}
\renewcommand{\t}{\mathfrak{t}}
\newcommand{\Z}{\mathsf{Z}}

\DeclareMathOperator{\Hilb}{Hilb}

\DeclareMathOperator{\ch}{ch}
\DeclareMathOperator{\Hom}{Hom}
\newcommand{\RHom}{R\mathscr{H}om}
\DeclareMathOperator{\rk}{rk}
\newcommand{\Pic}{\mathrm{Pic}}
\newcommand{\pr}{\mathrm{pr}}
\DeclareMathOperator{\Td}{Td}
\newcommand{\vir}{\mathrm{vir}}

\newcommand{\VW}{\mathrm{VW}}

\newcommand{\fN}{\mathfrak{N}}

\newtheorem{result}{Theorem}
\renewcommand*{\theresult}{\Alph{result}}

\numberwithin{equation}{section}
\newtheorem{theorem}[equation]{Theorem}

\newtheorem{conjecture}[equation]{Conjecture}
\newtheorem{corollary}[equation]{Corollary}
\newtheorem{definition}[equation]{Definition}
\newtheorem{lemma}[equation]{Lemma}
\newtheorem{proposition}[equation]{Proposition}

\theoremstyle{remark}
\newtheorem{notation}[equation]{Notation}
\newtheorem{example}[equation]{Example}
\newtheorem{remark}[equation]{Remark}

\makeatletter
\def\l@subsection{\@tocline{2}{0pt}{2pc}{6pc}{}}
\makeatother

\begin{document}


\begin{abstract}
We study the monopole contribution to the refined Vafa-Witten invariant, recently defined in \cite{MT}. We apply the results of \cite{GT2} to prove a universality result for the generating series of contributions of Higgs pairs with $1$-dimensional weight spaces. For prime rank, these account for the entire monopole contribution by a theorem of Thomas. We use toric computations to determine part of the generating series, and find agreement with the conjectures of \cite{GK} for rank $2$ and $3$.
\end{abstract}
\maketitle

\tableofcontents

\clearpage

\section{Introduction}
\subsection{Vafa-Witten invariants}
In \cite{TT}, Yuuji Tanaka and Richard Thomas proposed a definition of an $\mathrm{SU}(r)$ Vafa-Witten invariant \cite{VW}. Let $(S,H)$ be a polarized smooth complex surface with canonical bundle $\omega_S$. A \emph{Higgs pair} is a pair
\[(E,\phi)  \quad \text{with} \quad E\in \mathrm{Coh}(S), \, \phi \colon E \rightarrow E \otimes \omega_S\,.\]
Choose a rank $r$, Chern classes $c_1,c_2$ on $S$, and a line bundle $M$ on $S$ with $c_1(M) = c_1$. Assume that $r,c_1$ and $c_2$ are chosen in such a way that stability and semistability of Higgs pairs coincide (see Section \ref{SecStab}). Let
\[\N_{r,M,c_2}^\bot = \{(E,\phi)\mid \mathrm{tr\,} \phi = 0, \rk(E)=r, \det E \cong M, c_2(E) = c_2 \}\]
be the moduli space of Gieseker stable trace free Higgs pairs with fixed determinant. In \cite{TT} a symmetric perfect obstruction theory on $\N_{r,M,c_2}^\bot$ is constructed. Its dual complex is given by the cone
\begin{equation}\label{EqCone}
\RHom_\pi(E,E)_0\xrightarrow{[\ ,\phi]}\RHom_\pi(E,E \otimes \omega_S)_0\rightarrow T\,,
\end{equation}
where $(E,\phi)$ is a universal Higgs pair on $ \N_{r,M,c_2}^\bot\times S$ and
\[\pi\colon \N_{r,M,c_2}^\bot\times S \rightarrow  \N_{r,M,c_2}^\bot\]
denotes the projection.
The $\CC^*$-action on $ \N_{r,M,c_2}^\bot$, which is given by scaling the Higgs field, can be lifted to an equivariant structure on $E$. It gives rise to a localized virtual class, which is used to define the Vafa-Witten invariant by
\begin{equation}\label{EqDefVW}
\VW_{r,c_1,c_2}(S) = \int_{\left [ (\N_{r,M,c_2}^\bot )^{\CC^*}\right ]^{\vir}} \frac{1}{e(N^\vir)}\,,
\end{equation}
where $N^\vir$ is the virtual normal bundle to $ (\N_{r,M,c_2}^\bot )^{\CC^*}$ in $\N_{r,M,c_2}^\bot$, and $e(N^\vir)$ denotes its equivariant Euler class.

A Higgs pair $(E,\phi)$ in the fixed locus $(\N_{r,M,c_2}^\bot )^{\CC^*}$ can be equipped with a $\CC^*$-action, and hence decomposes into weight spaces. We may assume the 0 is the highest weight appearing in the decomposition. As explained in \cite{TT}, the Higgs field acts with weight $-1$. Hence we can write
\begin{align*}
E &= \bigoplus_{i=0}^k E_i \otimes \t^{-i}\\
\phi &= (\phi_1,\ldots,\phi_k)\colon E \rightarrow E\otimes \omega_S \otimes \t\,,
\end{align*}
where the $E_i$ are torsion free sheaves of rank $r_i$ and $\phi$ decomposes into maps
\[\phi_i\colon E_{i-1}\rightarrow E_i\otimes \omega_S \otimes \t \quad \text{for} \quad i = 1,\ldots,k\,.\]
We will write
\[\M_{(r_0,\ldots,r_k)} = \M_{(r_0,\ldots,r_k),c_1,c_2} \subset (\N_{r,M,c_2}^\bot )^{\CC^*}\]
for the open and closed locus of Higgs pairs with weight spaces of dimensions $r_0,\ldots,r_k$. The locus
\[\M_{(r)} = \left\{(E,\phi)\in (\N_{r,M,c_2}^\bot )^{\CC^*} \mid \phi =0\right\}\]
is called the \emph{instanton branch} \cite{GK}. It is isomorphic to the moduli space of torsion free rank $r$ sheaves, and its contribution to the Vafa-Witten invariant is the (localized) virtual Euler characteristic (up to a sign). Its complement in the $\CC^*$-fixed locus is called the \emph{monopole branch}. In this paper, we will discuss the contribution of the locus $\M_{1^r} = \M_{(1\cdots 1)}$ of Higgs pairs with $1$-dimensional weight spaces to the monopole branch. As an application of \cite{GT2}, we will describe the structure of the generating series of the contributions of $\M_{1^r}$ to the Vafa-Witten invariant, and compute them in some cases.

In \cite{MT} (see also \cite{T}), Maulik and Thomas define a refined version of the Vafa-Witten invariant.
It is a rational function in $\sqrt y$, rather than a rational number. It specializes to the unrefined invariant at $y=1$. The instanton contribution to the refined Vafa-Witten invariant is given, up to a sign and a power of $y$, by the $\chi_y$-genus \cite{FG} of the component $\M_{(r)}$, which refines the virtual Euler characteristic \cite{GK17}. We will discuss the contribution of $\M_{1^r}$ to the refined invariant.

\subsection{Nested Hilbert schemes}
Fix a rank $r$. For an $r$-tuple of non-negative integers $n = (n_0,\ldots,n_{r-1})$, and an $(r-1)$-tuple $\beta = (\beta_1,\ldots,\beta_{r-1})$ of classes in $H^2(S,\ZZ)$, let
\[S^{[n_i]} \coloneqq \Hilb^{n_i}(S)\]
denote the Hilbert schemes of $n_i$ points on $S$, and let $\Hilb_{\beta_i}(S)$ be the Hilbert schemes of curves on $S$ with class $\beta_i$. We will also write
\[\Hilb_\beta^{n}(S) \coloneqq S^{[n_0]}\times \cdots \times S^{[n_{r-1}]}\times \Hilb_{\beta_1}(S)\times\cdots\times \Hilb_{\beta_s}(S)\,.\]
The \emph{nested Hilbert scheme}
\begin{equation}\label{EqSmoAm}
i:S_\beta^{[n]} \hookrightarrow \Hilb_\beta^{n}(S)
\end{equation}
is defined as the incidence locus
\[\{I_0,\ldots,I_{r-1},C_1,\ldots,C_{r-1} \mid I_{i-1}(-C_i) \subset I_i\,\}\,.\]

The nested Hilbert schemes are studied in \cite{GSY}, in which a perfect obstruction theory is constructed. Write $\I^{[n_i]}$ for the universal ideal sheaf on $S^{[n_i]} \times S$, and let
\[\mathcal{D}_i\rightarrow\Hilb_{\beta_i}(S)\times S\]
be the universal curve with class $\beta_i$. Finally, write
\[\pi\colon S_\beta^{[n]} \times S\rightarrow S_\beta^{[n]}\]
for the projection.

\begin{theorem}\textup{\cite{GSY}}\label{TheVir}
The nested Hilbert scheme $S^{[n]}_\beta$ admits a perfect obstruction theory, the dual of which is given by a cone on
\[\left(\bigoplus_{i=0}^{r-1} \RHom_\pi(\I^{[n_i]},\I^{[n_i]})\right)_0 \rightarrow \bigoplus_{i=1}^{r-1}\RHom_\pi\left(\I^{[n_{i-1}]},\I^{[n_i]}(\mathcal{D}_i)\right)\,,\]
in which the LHS is the kernel of the trace map
\[\bigoplus_{i=0}^{r-1} \RHom_\pi(\I^{[n_i]},\I^{[n_i]})\rightarrow R\pi_*\O_S\,.\]
\end{theorem}

In \cite{GT2}, Gholampour and Thomas give another construction of the perfect obstruction theory, using virtual resolutions of degeneracy loci of complexes. Moreover, they give a formula for the induced virtual class in the ambient space~\eqref{EqSmoAm}. We will give the statement in the following restricted setting.

Let $S$ be a surface satisfying
\[H^1(\O_S) = 0 \quad \text{and} \quad p_g(S)>0\,.\]
For $i=0,\ldots, r-1$, let $\O_S(\beta_i)$ be the line bundle with $c_1(\O_S(\beta_i)) = \beta_i$, so we have
\[\Hilb_{\beta_i}(S) = |\O_S(\beta_i)| \coloneqq \PP(H^0(\O_S(\beta_i)))\,.\]
We will write
\[\mathcal{F} (\beta_i) = \mathcal{F} \otimes \O_S(\beta_i)\]
for any sheaf $\mathcal{F}$ on $S$.

\begin{theorem}\textup{\cite[Theorem~5.6]{GT2}}\label{TheComp}
After push-forward by $i$, the virtual class of $S_\beta^{[n]}$ is given by
\begin{align*}
i_*[S_\beta^{[n]}]^{\vir}
	= {} & \prod_{i=1}^{r-1} e\left(R\pi_*\O_S(\beta_i) - \RHom_\pi\left(\I^{[n_{i-1}]},\I^{[n_i]}(\beta_i)\right)\right) \\
	&  \cap \left[S^{[n_0]} \times \cdots \times S^{[n_{r-1}]}\right] \times \SW(\beta_1)\times \cdots \times \SW(\beta_{r-1})\\
	\in {} & A_{n_0+n_s}\left(\Hilb_\beta^{n}(S)\right)
\end{align*}
in which
\[\SW(\beta_i) \in A_0(|\O_S(\beta_i)|)\cong \ZZ\]
is the Seiberg-Witten invariant of $\beta_i$, considered as a $0$-cycle.
\end{theorem}

\begin{remark} \label{RemOkVan}
We write
\[e\left(R\pi_*\O_S(\beta_i) - \RHom_\pi\left(\I^{[n_{i-1}]},\I^{[n_i]}(\beta_i)\right)\right)\]
for $(n_{i-1} + n_i)$th Chern class of the $K$-theory class in the brackets, which behaves in some sense like a rank $n_{n-1}+n_i$ vector bundle. E.g., by the generalized Carlsson-Okounkov vanishing \cite{GT}, its Chern classes vanish beyond its rank.
\end{remark}

\begin{remark}
For the definition and some basic properties of Seiberg-Witten classes of algebraic surfaces with $H^1(S)=0$ and $p_g>0$, we refer to \cite[Section 6.3.1]{Mo} or \cite[Section 4]{GSY2}.
\end{remark}

\begin{remark}
It is Theorem~\ref{TheComp} that allows us to compute the $\M_{1^r}$ contributions to the Vafa-Witten invariant. A large part of our paper should be seen as an application of the result by Gholampour and Thomas.
\end{remark}

\subsection{Results}
The moduli space $\M_{1^r}$ is a union of nested Hilbert schemes $S_\beta^{[n]}$ \cite{GSY2,TT}. Moreover, the $\CC^*$-localized virtual class from \cite{TT} agrees with the virtual class from Theorem \ref{TheVir}. It follows that the contribution of each component $\M_{1^r}$ to the Vafa-Witten invariant is \emph{topological} \cite{GT2}. The observation that the generating series of these contributions is multiplicative cf. \cite{Goe}, leads to the following result.

\begin{notation}
We will write
\[\VW_{1^r,c_1,c_2}(S,y)\]
for the contribution of $\M_{1^r} = \M_{1^r,c_1,c_2}$ to the refined Vafa-Witten invariant of~\cite{MT}, and
\[\Z_{S,r,c_1}(q,y) = \frac{q^{\frac{1-r}{2r} c_1^2}}{\#H^2(S,\ZZ)[r] } \, \sum_{c_2\in\ZZ} \VW_{1^r,c_1,c_2}(S,y)\, q^{c_2}\]
for the generating series of such contributions. Here
\[H^2(S,\ZZ)[r] \coloneqq \ker \left(H^2(S,\ZZ) \xrightarrow{\cdot r} H^2(S,\ZZ)\right)\]
denotes the $r$-torsion subgroup of $H^2(S,\ZZ)$.
\end{notation}

\begin{result}\label{Theorem1}
Fix a rank $r\geq1$. There are universal Laurent series
\[A,\,B,\,C_{ij} \in \QQ(\sqrt{y})(\!(q^\frac{1}{2r})\!) \quad 1\leq i\leq j\leq r-1 \,,\]
depending \emph{only} on $r$, such that for any surface $S$ with $H^1(\O_S) = 0$ and $p_g(S)>0$, and any class $c_1\in H^2(S,\ZZ)$ such that semistability of Higgs pairs implies stability for all $c_2$, we have
\begin{equation} \label{EqGenSer}
\Z_{S,r,c_1}(q,y) = 
	A^{\chi(\O_S)}
	B^{K_S^2} \,
\sum_\beta
	 \SW(\beta^1)\cdots\SW(\beta^{r-1})
	\prod_{i\leq j}C_{ij}^{\beta^i\beta^j}
\end{equation}
where the second sum is taken over tuples $\beta = (\beta^1,\ldots,\beta^{r-1})\in (H^2(S,\ZZ))^{r-1}$ with
\[c_1 \equiv \sum_i i\beta^i \mod rH^2(S,\ZZ)\,.\]
\end{result}
\begin{remark}
The condition $H^1(\O_S)=0$ in Theorem \ref{Theorem1} is superfluous. However, for expository reasons, we will work with this condition throughout the paper. Moreover, \eqref{EqGenSer} can be shown to hold for \emph{all} $c_1$ when we extend the definition of the LHS to the semistable case. This will be subject of future work.
\end{remark}

\begin{remark}
For odd rank $r$, the Laurent series have coefficients in $\QQ(y)$, rather than in $\QQ(\sqrt y )$ (see Proposition \ref{PropSqrt}).
\end{remark}

The following corollary is implicitly in the statement of Theorem \ref{Theorem1}.

\begin{corollary}
Fix a rank $r$, and let $S$ be a surface with $H^1(\O_S) = 0$ and $p_g(S)>0$. Let $c_1$ be a Chern classes, for which semistability implies stability for all $c_2$. Then $\Z_{S,r,c_1}(q,y)$ is independent of the choice of a polarization of the surface $S$.
\end{corollary}

We will define the Laurent series in the theorem explicitly in terms of tautological integrals over products of Hilbert schemes of points on the surface $S$ (see Sections~\ref{SecTrace},~\ref{SecUniv}~and~\ref{SecProof}). Although for surfaces with $\deg(K_S) < 0$, the locus $\M_{1^r}$ is empty by stability, the Hilbert schemes and the integrals are still defined. We will prove universality of these integrals for \emph{all} surfaces (Proposition \ref{PropUniv}). As usual \cite{Goe}, the coefficients of the power series can be determined by evaluating these integrals on $\PP^2$ and $\PP^1\times \PP^1$, where we have access to toric methods, as we explain in Section~\ref{SecComp}.

\subsection{G\"ottsche-Kool conjectures}
In \cite{GK}, Lothar G\"ottsche and Martijn Kool conjecture a formula for the generating series of the $\chi_y$-genus of the instanton branch for rank $2$ and $3$. Moreover they conjecture, motivated by $S$-duality \cite{VW}, that the generating series of refined Vafa-Witten invariants has modular properties that relate the contributions of the instanton branch to those of the monopole branch. Using this, they give a conjectural formula for the contribution of the monopole branch to refined Vafa-Witten invariants of rank $2$ and $3$. For rank $2$, their conjectures refine the predictions in the physics literature \cite{VW}.

The formulas of \cite{GK} that predict the monopole contributions to the Vafa-Witten invariants in rank 2 and 3, have precisely the structure of the generating series \eqref{EqGenSer} of the $\M_{1^r}$ contributions. This suggests that $\M_{1^r}$ accounts for the entire monopole contribution.

\begin{conjecture}\label{ConjMono}
For $S$ and $c_1$ as in Theorem~\ref{Theorem1}, and $r$ prime, we have
\[\VW_{r,c_1,c_2}^{\mathrm{monopole}}(S,y) = \VW_{1^r,c_1,c_2}(S,y)\]
for all $c_2\in \ZZ$.
\end{conjecture}

The conjecture has now been proved by Thomas in \cite{T}.

\begin{theorem}[Thomas]\label{TheThoVan}
Conjecture \ref{ConjMono} holds.
\end{theorem}

It follows that Theorem \ref{Theorem1} and Theorem \ref{TheThoVan} prove the \emph{structure} of \cite[Conjecture 1.5]{GK}, generalized to arbitrary prime rank. The rank $2$ and $3$ conjectures of \cite{GK} give the universal series  appearing in the formula explicitly in terms of functions
\[\phi_{-2,1}(x,y),\, \Delta(x),\, \Theta_{A_2,(1,0)}(x,y),\, \eta(x),\,\theta_2(x,y),\,\theta_3(x,y),\,\text{and}\,\, W_\pm(x,y),\]
which we give in Appendix \ref{AppFun}.  The following conjectures imply \cite[Remark 1.7 and Conjecture 1.5]{GK}.
\begin{notation}
In order to emphasize the dependency on $r$, we will write $A^{(r)}, B^{(r)},\ldots$ for the series appearing in Theorem \ref{Theorem1}.
\end{notation}

\begin{conjecture}\label{ConjR2}
For rank $2$, the universal series appearing in Theorem \ref{Theorem1}, and defined in Section~\ref{SecProof}, are given by
\begin{align*}
A^{(2)}(y)
	&= \frac{y^{\frac{1}{2}}-y^{-\frac{1}{2}}}{\phi_{-2,1}(q^2,y^2)^\frac{1}{2}\, \tilde\Delta(q^2)^\frac{1}{2}}\,,\\
B^{(2)}(y)
	&= \frac{\tilde\eta(q)^2}{\theta_3(q,y)}\,,\\
C^{(2)}_{11}(y)
	&= \frac{-\theta_3(q,y)}{\theta_2(q,y)}\,.
\end{align*}
\end{conjecture}

\begin{conjecture}\label{ConjR3}
For rank $3$, we have
\begin{align*}
A^{(3)}(y)
	&= \frac{y^\frac{1}{2} - y^{-\frac{1}{2}}}{\phi_{-2,1}(q^3,y^3)^\frac{1}{2}\, \tilde\Delta(q^3)^\frac{1}{2}} \,,\\
B^{(3)}(y)
	&= \frac{\tilde\eta(q)^3\, W_-(q^\frac{1}{2},y)}{\Theta_{A_2,(1,0)}(q^\frac{1}{2},y)}\,,\\
C_{12}^{(3)}(y)
	&= W_+(q^\frac{1}{2},y)\, W_-(q^\frac{1}{2},y)\,,\\
C_{11}^{(3)}(y)
	&= C_{22}^{(3)}(y) \\
	&= \frac{1}{W_-(q^\frac{1}{2},y)}\,.
\end{align*}
\end{conjecture}

\subsection{Toric computations}

As remarked before, the universality result of Proposition \ref{PropUniv} allows us to determine the first few terms of the power series of Theorem \ref{Theorem1} by toric computations. I~implemented the Atiyah-Bott localization formula for the surfaces $\PP^2$ and $\PP^1\times\PP^1$ in Sage \cite{Sage}  and found agreement with Conjectures \ref{ConjR2} and \ref{ConjR3}.

Define multiplicative subgroups
\[U^{(r)}_N \coloneqq 1 + q^N \,\QQ(y^\frac{1}{2})[[q]]\subset \QQ(y^\frac{1}{2})(\!(q^\frac{1}{2r})\!)^*\,.\]
for all $r,N\geq 1$, and consider series
\begin{align*}
P
	& = c q^{\frac{z}{2r}}(1+p_1 q + p_2 q^2 + \ldots) \quad \text{and}\\
P'
	& = c' q^{\frac{z'}{2r}}(1+p'_1 q + p'_2 q^2 + \ldots)
\end{align*}
with $c, c' \in \QQ(\sqrt y)^*$, $p_i, p'_i \in \QQ(\sqrt y)$ and $ z, z'\in \ZZ$. The Laurent series appearing in Theorem~\ref{Theorem1} and Conjectures \ref{ConjR2} and \ref{ConjR3} are all of this form. Then we have
\[P \equiv P' \mod{}U^{(r)}_{N+1}  \quad  \text{(i.e. }P'P^{-1} \in U^{(r)}_{N+1}\text{)}\]
if and only if
\[c = c', \quad z =  z', \quad \text{and}\quad p_1 = p'_1,\ldots,\, p_N = p'_N\,.\]

\begin{result} \label{Theorem2}
Let $S$ be a surface with $H^1(\O_S) = 0$ and $p_g(S)>0$.
The rank $2$ conjectures of \cite{GK} correctly predict the first 15 terms of the universal series of Theorem \ref{Theorem1}. The rank $3$ conjectures correctly predict the first 11 terms.
In other words, the equations of Conjecture \ref{ConjR2} hold modulo $U^{(2)}_{15}$,
and the equations of Conjecture \ref{ConjR3} hold modulo $U^{(3)}_{11}$.
\end{result}

\subsection{K3 surfaces}
Let $S$ be a K3 surface. Then $0\in H^2(S,\ZZ)$ is the only Seiberg-Witten basic class of $S$, and for $c_1 = 0$, equation \eqref{EqGenSer} becomes
\[``\Z_{S,r,0}(q,y)" = \left(A^{(r)}\right)^2\,.\]
Note that in our setting, the left-hand side has not been defined for $r>1$, due to the existence of strictly semistable sheaves. Hence we cannot apply Theorem \ref{Theorem1} directly to determine the power series $A^{(r)}$. We can, however, evaluate on $S$ the tautological integrals that are used to define the universal series $A^{(r)}$. This leads to the following result, which we prove in Section~\ref{SecProofC}.

\begin{result}\label{Theorem4}
We have
\begin{equation}\label{EqK3GenSer}
A^{(r)}(y) = \frac{y^\frac{1}{2} - y^{-\frac{1}{2}}}{\phi_{-2,1}(q^r,y^r)^\frac{1}{2}\, \tilde\Delta(q^r)^\frac{1}{2}}
\end{equation}
for any $r\geq1$.
\end{result}

\begin{remark}
Refined Vafa-Witten invariants of K3 surfaces have been computed in \cite{T}. In a future paper, we will extend Theorem~\ref{Theorem1} to the semistable case, so
Theorem~\ref{Theorem4} will follow also from the results of loc.~cit..
\end{remark}

\subsection{A special case}

Let $S$ be a surface with $H^1(\O_S) = 0$ and $p_g(S)>0$, and assume the Picard group
\[\Pic(S) = \ZZ\cdot [C]\]
of $S$ is generated by a smooth very ample canonical curve $C\in |K_S|$. Let $c_1 = K_S$. For rank $3$, only  $\beta = (K_S,0)$ contributes to the right-hand side of \eqref{EqGenSer} (see Lemma \ref{LemBasic}). In rank $2$, and in a slightly more general setting \cite{TT}, the only contribution is given by $\beta = (K_S)$. By Theorem \ref{Theorem1}, we have
\begin{equation}\label{EqBasic}
\Z_{S,r,K_S}(q,y)= \left(-A^{(r)}(y)\right)^{\chi(\O_S)} \left(B^{(r)}(y)C_{11}^{(r)}(y)\right)^{K_S^2}
\end{equation}
for $r=2,3$. Here we have used the equation
\begin{equation}\label{EqSWKS}
\SW(K_S) = (-1)^{\chi(\O_S)}
\end{equation}
by e.g.\ \cite[Proposition 6.3.4]{Mo}. In this setting, our computations are slightly faster, and we find the following result.

\renewcommand\theresult{B\ensuremath{'}}
\begin{result} \label{Theorem3}
Let $S$ be a surface with $H^1(\O_S) = 0$ and $p_g(S)>0$, and assume that the Picard group of $S$ is generated by a smooth very ample canonical curve. Then we have
\begin{align*}
\Z_{S,2,K_S}(q,y)
	& \equiv \left(\frac{-(y^{\frac{1}{2}}-y^{-\frac{1}{2}})}{\phi_{-2,1}(q^2,y^2)^\frac{1}{2} \tilde\Delta(q^2)^\frac{1}{2}}\right)^{\chi(\O_S)}
	\left(\frac{-\tilde\eta(q)^2}{\theta_2(q,y)}\right)^{K_S^2}  \mod U^{(2)}_{17}\,, \\
\Z_{S,3,K_S}(q,y)
	& \equiv \left(\frac{-(y^\frac{1}{2} - y^{-\frac{1}{2}})}{\phi_{-2,1}(q^3,y^3)^\frac{1}{2} \tilde\Delta(q^3)^\frac{1}{2}}\right)^{\chi(\O_S)}
	\left(\frac{\tilde\eta(q)^3}{\Theta_{A_2,(1,0)}(q^\frac{1}{2},y)}\right)^{K_S^2}  \mod U^{(3)}_{14} \,.
\end{align*}
\end{result}


For $S$ a surface as in Theorem \ref{Theorem3}, and rank $r=2$, the moduli space $\M_{1^2,K_S,c_2}$ is smooth for $c_2\leq 3$. In \cite{TT} and \cite{T}, this is used to compute the Vafa-Witten invariant by direct intersection-theoretic calculations. The rank $2$ equation of Theorem \ref{Theorem3} is proved modulo $U^{(2)}_3$ in \cite{T}. In \cite{TT}, it is proved modulo $U^{(2)}_4$ in the unrefined case.

For rank $3$, the moduli space $\M_{1^3,K_s,c_2}$ is smooth if and only if $c_2\leq2$ (see Proposition \ref{PropSmoothMod}). This allows us to compute the Vafa-Witten invariants by the methods of \cite{TT,T}. As a result, we obtain an \emph{alternative proof}, by direct calculations, for the rank $3$ equation of Theorem \ref{Theorem3}, modulo $U^{(3)}_3$.

\subsection{Acknowledgement}
I thank Martijn Kool and Richard Thomas for sharing early drafts of their papers \cite{GK} and \cite{GT2} with Lothar G\"ottsche and Amin Gholampour respectively. I thank Amin Gholampour, Richard Thomas, and especially my PhD supervisor Martijn Kool for many useful discussions and suggestions.

This paper is partially based on work that was done at the MSRI in Berkeley, CA, during the Spring 2018 semester.

\section{The moduli space}\label{SecMod}
Let $S$ be a smooth projective surface with $H^1(\O_S) = 0$, and fix a rank $r$. As mentioned in the introduction, the locus $\M_{1^r}$ of Higgs pairs with $1$-dimensional weight spaces is a union of nested Hilbert schemes. In this section, we will introduce some notation and describe universal Higgs pairs over the connected components.

Write $s \coloneqq r-1$ and let $L = (L_0,\ldots,L_s)$ be an $r$-tuple of line bundles on $S$.
\begin{notation}\label{Notation}
Define classes
\begin{align*}
\beta_i
	= {} & c_1(L_{i}\otimes L_{i-1}^* \otimes \omega_S) \\
	\in {} & H^2(S,\ZZ)
\end{align*}
for $i=1,\ldots,s$, and write
\[\beta = \beta(L) =  (\beta_1,\ldots,\beta_s)\,. \]
We will also write
\[\beta^i = K_S - \beta_i \]
for $i=1,\ldots,s$ and an $s$-tuple $\beta = (\beta_1,\ldots,\beta_s) \in (H^2(S,\ZZ))^s$. In particular, when $\beta = \beta(L)$, we have
\[\beta^i =  c_1(L_{i-1}\otimes L_{i}^*)\]
for $i = 1,\ldots,s$.
\end{notation}

\begin{remark}
We will use the convention
\[s \coloneqq r-1\]
throughout the paper. Furthermore, $L$ will always denote an $s+1$-tuple of line bundles on $S$, and $\beta$ an $s$-tuple of classes in $H^2(S,\ZZ)$.
\end{remark}

Consider the product of complete linear systems
\begin{align*}
\Hilb_\beta(S)
	& = \Hilb_{\beta_1}(S)\times \cdots\times \Hilb_{\beta_{s}}(S)\\
	& = |\O_S(\beta_1)|\times\cdots\times |\O_S(\beta_s)|\,.
\end{align*}
and write
\[
\begin{tikzcd}
\vert \O_S(\beta_i)\vert \times S &  \Hilb_\beta(S) \times S \arrow[l,swap,"\pr_i"]\arrow[d,"\pi"] \arrow[r,"q"] & S \\
 & \Hilb_\beta(S) & 
\end{tikzcd}
\]
for the projections, where $i=1,\ldots,s$. We will write $\O_{\beta_i}(1)$ for the canonical line bundle on $|\O_S(\beta_i)|$.

Define the following line bundles on $\Hilb_\beta (S)\times S$:
\begin{align*}
\L_0 &\coloneqq L_0 \\
\L_1 &\coloneqq L_1 \otimes \pr_1^*\O_{\beta_1}(1) \\
&\cdots \\
\L_s &\coloneqq L_s \otimes \pr_1^*\O_{\beta_1}(1)\otimes\cdots\otimes \pr_s^*\O_{\beta_s}(1)\,.
\end{align*}
The tautological sections
\[\O_{|\O_S(\beta_i)|\times S} \rightarrow \O_{\beta_i}(1) \boxtimes \O_S(\beta_i)\]
induce maps
\[\phi_{\L,i}\colon\L_{i-1}\rightarrow \L_i\otimes q^*\omega_S\]
for $i=1,\ldots,s$.

Let $\t$ be an equivariant parameter for the trivial $\CC^*$-action on a point. Define the locally free sheaf
\[E_\L \coloneqq (\L_0\otimes \t^0) \oplus \ldots \oplus (\L_s\otimes \t^{-s})\,.\]
The maps $\phi_{\L,i}$ define a $\CC^*$-equivariant Higgs field
\[\phi_\L = (\phi_{\L,1},\ldots,\phi_{\L,s}):E_\L\rightarrow E_\L\otimes\omega_S\otimes\t\,.\]

Now choose non-negative integers $n = (n_0,\ldots,n_s)$ and write
\[\Hilb^{n}_\beta(S) = S^{[n_0]}\times\cdots\times S^{[n_s]}\times \Hilb_\beta(S)\]
as in the introduction. Let $\I^{[n_i]}$ denote the universal ideal sheaf on $S^{[n_i]}\times S$. Define the following sheaf on $\Hilb^{n}_\beta(S) \times S$, suppressing obvious pull-backs:
\[E_\L^{[n]} \coloneqq (\L_0\otimes\I^{[n_0]} \otimes \t ^0) \oplus \ldots \oplus (\L_s\otimes \I^{[n_s]} \otimes \t^{-s})\,.\]
The nested Hilbert scheme is by definition the maximal subscheme
\[i\colon S_\beta^{[n]}\hookrightarrow \Hilb^{n}_\beta(S)\]
such that $\phi_\L$ restricts to a Higgs field
\[\phi_\L^{[n]}:E_\L^{[n]}\rightarrow E_\L^{[n]}\otimes\omega_S\otimes\t\,.\]
\begin{remark}
Throughout the paper, the letter $n$ is reserved for $s+1$-tuples of non-negative integers, and $\Hilb_\beta^n(S)$ will always denote a product of Hilbert schemes as above.
\end{remark}
\begin{proposition}[\cite{GSY}, \cite{TT}]\label{PropMod}
Every connected component of $\M_{1^r}$ is uniquely represented by a triple
\[(S_\beta^{[n]},E_\L^{[n]},\phi_\L^{[n]})\,,\]
as constructed above.
\end{proposition}

By Proposition \ref{PropMod}, the connected components of $\M_{1^r}$ are naturally indexed by tuples $L = (L_0,\ldots,L_s)$ and $n=(n_0,\ldots, n_s)$. 

\begin{definition}
We denote a connected component of $\M_{1^r}$ represented by a triple $(S_\beta^{[n]},E_\L^{[n]},\phi_\L^{[n]})$ by $\M_L^{[n]}$.
\end{definition}

Obviously, not every pair $(L,n)$ corresponds to a connected component of $\M_{1^r}$. The nested Hilbert scheme might be empty, or the Higgs pairs in the family $(E_\L^{[n]},\phi_\L^{[n]})$ might be unstable. The other restriction is the Chern data of the Higgs pairs. We will address stability in Section \ref{SecStab}. We finish this section with a lemma regarding the second issue.


\begin{lemma}\label{LemmaChern}
The total Chern class (in cohomology) of any fibre $E$ of $E_\L^{[n]}$ over $S_\beta^{[n]}$ is given by
\begin{align*}
c(E)
	& = 1 + c_1 + |n|\cdot pt + \sum_{0\leq i<j \leq s} c_1(L_i)c_1(L_j) \\
	& = 1 + (s+1)c_1(L_s) + \sum_{i=1}^s i \beta^i\\
	& \quad + |n|\cdot pt + \frac{s}{2 (s+1)}c_1^2 - \sum_{1\leq i<j\leq s} \frac{i(s+1-j)}{s+1}\beta^i\beta^j - \sum_{1\leq i\leq s} \frac{i(s+1-i)}{2(s+1)}(\beta^i)^2\,,
\end{align*}
where
\begin{align*}
c_1
	& = c_1(L_0)+\ldots+c_1(L_s)\,, \\
|n|
	& = n_0+\ldots+n_s\,,
\end{align*}
and $pt$ denotes the Poincar\'e dual of the homology class of a point.
\end{lemma}
\begin{proof}
This is a straight forward computation. For the second equation, note that we have
\[(s+1)c_1(L_i) = \sum_{k=1}^i -k \beta^k + \sum_{k=i+1}^s (s+1-k) \beta^k + c_1\]
for $i = 0,\ldots,s$. Substituting this into
\[\sum_{0\leq i<j \leq s} c_1(L_i)c_1(L_j)\]
and interchanging sums gives the result.
\end{proof}

\section{Stability}\label{SecStab}
By Proposition \ref{PropMod}, the connected components of $\M_{1^r}$ are isomorphic to nested Hilbert schemes $S_\beta^{[n]}$, with
\[\beta = (\beta_1,\ldots,\beta_s) \quad \text{and} \quad n = (n_0,\ldots,n_s)\,\]
tuples of divisor classes on $S$ and integers respectively. The Hilbert scheme is empty if and only if one of the $\beta_i$'s is not effective, or $\beta_i=0$ and $n_{i-1}<n_i$ for some $i$. Obviously, the virtual class of the nested Hilbert scheme vanishes in this case.

We will give dual conditions on $\beta$ and $n$, which hold whenever the Higgs pairs parametrized by $S_\beta^{[n]}$ are Gieseker unstable, and which in turn imply the vanishing of the virtual class. We recall the definition of stability of Higgs pairs.

\begin{definition}
Let $H$ be a polarization of the surface $S$. A Higgs pair $(E,\phi)$ is called \emph{slope stable} (resp. \emph{slope semistable}) if
\[\frac{\deg(F)}{\rk(F)} < \frac{\deg(E)}{\rk(E)} \quad (\text{resp.} \quad \frac{\deg(F)}{\rk(F)} \leq \frac{\deg(E)}{\rk(E)})\]
for every $\phi$-invariant subsheaf $0\neq F\subsetneq E$ with $\rk(F)<\rk(E)$. It is called \emph{Gieseker stable} (resp. \emph{Gieseker semistable}) if we have inequalities of polynomials in $m$
\[\frac{\chi(F(mH))}{\rk(F)} < \frac{\chi(E(mH))}{\rk(E)} \quad (\text{resp.} \quad \frac{\chi(F(mH))}{\rk(F)} \leq \frac{\chi(E(mH))}{\rk(E)})\]
for every proper $\phi$-invariant subsheaf $0\neq F\subsetneq E$. By ``(semi)stable'', we will always mean Gieseker (semi)stable.
\end{definition}

Let $E = E_0 \oplus \ldots \oplus E_s$ be a sum of torsion free rank $1$ sheaves, equipped with a Higgs field
\[\phi=(\phi_1,\ldots,\phi_s) \colon E \rightarrow E\otimes \omega_S\]
given by homomorphisms
\[\phi_i\colon E_{i-1} \rightarrow E_i \otimes \omega_S \quad \text{for } i=1,\ldots,s\,.\]
Note that all Higgs pairs in $\M_{1^r}$ are of this form.
\begin{lemma}\label{LemSSs}
Assume that $(E,\phi)$ is indecomposable, i.e.\ $\phi_i \neq 0$ for $i=1,\ldots,s$ and assume that
\[\deg(E_{i-1}) \geq \deg(E_i) \quad \text{for} \quad i=1,\ldots,s\,.\]
Then the pair $(E,\phi)$ is slope semistable. It is slope stable unless
\[\deg(E_0) = \ldots = \deg(E_s)\,.\]
\end{lemma}
\begin{proof}
Let $F\subset E$ be a $\phi$-invariant Higgs field. Let $j$ be maximal, such that
\begin{equation}\label{EqRankUp}
F\subset E_j\oplus\ldots\oplus E_s\,.
\end{equation}
I claim that $F$ has rank $s+1-j$. It follows that if $F$ is a destabilizing subsheaf, so is $E_j\oplus\ldots\oplus E_s$.

In order to prove the claim, consider the filtration
\[F = F^0 \supset \ldots \supset F^{s-j} \supset F^{s+1-j} = 0\]
of $F$, given by
\[F^i = K_S^{-i}\otimes \phi^{\circ i} (F)\,,\]
so we have
\[F^i \subset E_{j+i}\oplus\ldots\oplus E_s\]
for $i=0,\ldots,s+1-j$. Note that for $i = 0,\ldots,s-j$, by injectivity of $\phi_{j+i}\cdots\phi_{j+1}$, and by the choice of $j$, the composition
\[F^i \subset E_{j+i}\oplus\ldots\oplus E_s \rightarrow E_{j+i}\]
is non-zero, and hence its image has rank $1$, since $E_{j+i}$ is torsion-free. On the other hand, its kernel contains $F^{i+1}$. It follows that we have
\[\rk F > \rk F^1 > \ldots > \rk F^{s+1-j} = 0\]
and hence, $\rk(F) = s+1-j$ by \eqref{EqRankUp}, proving the claim.

It follows that $(E,\phi)$ is slope semistable if and only if
\[\frac{\sum_{i=j}^s \deg(E_i)}{s+1-j} \leq \frac{\sum_{i=0}^s \deg(E_i)}{s+1} = \frac{\deg(E)}{\rk(E)}\]
for $j=0,\ldots,s$.
This clearly holds when $\deg(E_i) \leq \deg(E_{i-1})$ for all $i$. Finally note that $(E,\phi)$ is slope stable if one of the inequalities is strict.
\end{proof}

The hypothesis of Lemma \ref{LemSSs} certainly holds when $c_1(E_{i-1}) - c_1(E_i)$ is effective for each $i$. In this case, the condition
\[\deg(E_0) = \ldots = \deg(E_s)\]
implies that
\[c_1(E_0) = \ldots = c_1(E_s)\,.\]
Although such a Higgs pair is not slope stable, it might still be Gieseker (semi)stable.

\begin{lemma}\label{LemSs}
Assume that $(E,\phi)$ is indecomposable, and assume that
\begin{align*}
c_1(E_{0})= \ldots = c_1(E_s)
\quad \text{and} \quad
c_2(E_{0}) \leq \ldots \leq c_2(E_s)\,.
\end{align*}
Then the pair $(E,\phi)$ is Gieseker semistable. It is Gieseker stable unless
\[c_2(E_0) = \ldots = c_2(E_s)\,.\]
\end{lemma}
\begin{proof}
The proof is similar to the proof of Lemma \ref{LemSSs}. Simply note that by Grothendieck-Riemann-Roch the hypothesis implies
\[\chi(E_{i-1}(m)) \geq \chi(E_i(m))\]
for $i=1,\ldots,s$, with equality whenever $n_{i-1} = n_i$.
\end{proof}

Now let $S$ be a surface with $p_g(S)>0$ and $H^1(\O_S) = 0$. Let $L_0,\ldots,L_s$ be line bundles on $S$ and let $n = (n_0,\ldots,n_s)$ be non-negative integers. Let $\beta = \beta(L)$ (and $\beta_i$ and $\beta^i$ for $i=1,\ldots,s$) be given as in Notation~\ref{Notation}, and consider the flat family of Higgs pairs $(E_\L^{[n]},\phi_\L^{[n]})$ over the base $S_\beta^{[n]}$, as defined in Section \ref{SecMod}.

In terms of $\beta$ and $n$, Lemma \ref{LemSSs} and Lemma \ref{LemSs} tell us that whenever the family $(E_\L^{[n]},\phi_\L^{[n]})$ is not Gieseker semistable, there is an $i\in\{1,\ldots,s\}$ such that the divisor class $\beta^i$ is not effective, or such that $\beta^i = 0$ and $n_{i-1}>n_i$ (compare to the introduction of this section!). As we will see in the following proposition, this suffices to show that we have $i_*[S_\beta^{[n]}]^\vir= 0$ in this case (recall that we write $i\colon S_\beta^{[n]}\hookrightarrow \Hilb_\beta^n(S)$).


\begin{proposition}\label{PropGSs}
Assume that
\[i_*[S_\beta^{[n]}]^{\vir} \neq 0 \,.\]
Then the family $(E_\L^{[n]},\phi_\L)$ of Higgs pairs is (Gieseker) semistable for any polarization of $S$. It is stable unless $L_0=\ldots=L_s$ and $n_0=\ldots=n_s$.
\end{proposition}
\begin{proof}
By Theorem \ref{TheComp}, equation \eqref{EqSWKS} and the hypothesis,  we have
\begin{align*}
\SW(\beta^1)\cdots\SW(\beta^s)
	& =(-1)^{s\cdot \chi(\O_S)} \SW(\beta_1)\cdots\SW(\beta_s) \\
	& \neq 0\,.
\end{align*}
It follows that $\beta^i \geq 0$ for $i=0,\ldots,s$, by definition of the Seiberg-Witten class.
By Lemma \ref{LemSSs}, the fibres of $(E_\L^{[n]},\phi_\L)$ are slope-stable, and hence Gieseker stable, unless $L_0=\ldots=L_s$. Assume the latter. By Lemma \ref{LemSs}, we need to show that $n_{i-1}\leq n_i$ for all $i$. Assume that $n_{i-1}>n_i$ for some $i$. Then the nested Hilbert scheme
\[i\colon S^{[n_i,n_{i-1}]}\hookrightarrow S^{[n_i]} \times S^{[n_{i-1}]}\]
is empty, and we have by Serre duality and Theorem \ref{TheComp} 
\begin{align} \label{EqFactor}
& e\Big(R\pi_*\omega_S - \RHom_\pi(\I^{[n_{i-1}]},\I^{[n_i]}\otimes {\omega_S})\Big) \\
& = (-1)^{n_{i-1} + n_i}  e\left(R\pi_*(\O_S) - \RHom_\pi(\I^{[n_i]},\I^{[n_{i-1}]})\right)\notag\\
& = (-1)^{n_{i-1} + n_i}  i_*[S^{[n_i,n_{i-1}]}]^\vir \notag\\
& = 0 \,.\notag
\end{align}
By the assumption $L_0=\ldots=L_s$, we have in particular $\beta_i=K_S$. It follows that the expression of Theorem~\ref{TheComp} has a factor \eqref{EqFactor}. We find $i_*[S_\beta^{[n]}]^\vir = 0$, which contradicts the hypothesis.
\end{proof}


Recall from Proposition \ref{PropMod}, that $\M_{1^r,c_1,c_2}$ is a union of Hilbert schemes $S_\beta^{[n]}$. We will also write $i$ for the morphism
\[i\colon \M_{1^r,c_1,c_2} \rightarrow \bigsqcup_{\beta,n} \Hilb_\beta^n(S)\,,\]
which is given on each connected component of $\M_{1^r,c_1,c_2}$ by the inclusion
\[i\colon S_\beta^{[n]} \rightarrow \Hilb_\beta^n(S)\,.\]
By the vanishing of Proposition \ref{PropGSs}, we can sum in the following proposition over all pairs $(L,n)$ or $(\beta,n)$, rather than the ones that correspond to connected components of  \emph{stable} Higgs pairs. In particular, the push-forward by $i$ of the virtual class does not depend on the polarization of the surface $S$.

\begin{proposition}\label{PropVirClass}
Let $S$ be a surface with $p_g(S)>0$ and $H^1(\O_S) = 0$. Fix $r$, $c_1$ and $c_2$ such that Gieseker semistability of Higgs pairs implies Gieseker stability. Then we have
\begin{equation}\label{EqVirCl}
\begin{split}
i_*[\M_{1^r,c_1,c_2}]^{\vir}
	& = \sum_{L,n} i_*[S_{\beta(L)}^{[n]}]^{\vir} \\
	& = \sum_{\beta,n} i_*[S_{\beta}^{[n]}]^{\vir}  \cdot  \#\ker \left(H^2(S,\ZZ) \xrightarrow{\cdot (s+1)} H^2(S,\ZZ)\right)
\end{split}
\end{equation}
where the sums are taken over
\begin{align*}
L,n \quad \text{with} \quad
\begin{split}
\quad c_1
	&= \sum_{i=1}^s c_1(L_i)\\
\quad c_2
	& = |n| + \sum_{0\leq i<j \leq s} c_1(L_i)c_1(L_j)\,;
\end{split}	\\
\intertext{and, respectively,}
\beta,n \quad \text{with}\quad
\begin{split}
\quad c_1
	& \equiv \sum_{i=1}^s i \beta^i \mod (s+1)H^2(S,\ZZ)\\
\quad c_2
	& = |n| + \frac{s}{2(s +1)}c_1^2 - \sum_{1\leq i<j\leq s} \frac{i(s+1-j)}{s+1}\beta^i\beta^j
\end{split}
\\
	& \quad - \sum_{1\leq i\leq s} \frac{i(s+1-i)}{2(s + 1)}(\beta^i)^2 \,.
\end{align*}
\end{proposition}
\begin{proof}
The $c_2$-conditions on the pairs $(L,n)$ and $(\beta,n)$ appearing in the sums are given by Lemma~\ref{LemmaChern}.
Moreover, it is easy to see that for an $s$-tuple of curve classes $\beta = (\beta_1,\ldots,\beta_s) \in \left(H^2(S,\ZZ)\right)^s$, there is a tuple of vector bundles $L = (L_0,\ldots, L_s)$ with
\[c_1 = \sum_{i=1}^s c_1(L_i) \quad \text{and} \quad \beta = \beta(L)\]
if and only if
\[\quad c_1 \equiv \sum_{i=1}^s i \beta^i \mod (s+1)H^2(S,\ZZ)\,.\]
Now assume that $S_\beta^{[n]}\cong \M_L^{[n]} \subset \M_{1^r,c_1,c_2}$ is a connected component. For a curve class
\[\gamma \in \ker \left(H^2(S,\ZZ) \xrightarrow{\cdot (s+1)} H^2(S,\ZZ)\right) \eqqcolon K\]
there is a connected component
\[S_\beta^{[n]}\cong \M_{L(\gamma)}^{[n]}\subset \M_{1^r,c_1,c_2}\,,\]
where
\[L(\gamma) = (L_0\otimes \O_S(\gamma), \ldots, L_s\otimes \O_S(\gamma))\,.\]
In fact, there is a $K$-torsor of connected components of $\M_{1^r,c_1,c_2}$ that are isomorphic to $S_\beta^{[n]}$. This expains the second equation of \eqref{EqVirCl}.

Finally, note that pairs $(\beta,n)$ for which the scheme $S_\beta^{[n]}$ is empty obviously do not contribute to the right-hand side of \eqref{EqVirCl}. By Proposition \ref{PropGSs}, the same holds for pairs $(\beta,n)$ for which $S_\beta^{[n]}$ parametrizes unstable sheaves.
\end{proof}
%


\section{Tautological integrals}\label{SecVWInt}
Choose line bundles $L = (L_0,\ldots, L_s)$ on $S$ and let $\beta = \beta(L) = (\beta_1,\ldots,\beta_s)$ and $\L = (\L_0,\ldots,\L_s)$ be defined as in Section \ref{SecMod}. Let $n=(n_0,\ldots,n_s)$ be non-negative integers. Recall that we write
\[E_\L^{[n]} = \L_0 \otimes \I^{[n_0]} \otimes \t^0 \oplus \ldots \oplus \L_s \otimes \I^{[n_s]} \otimes \t^{-s}\]
for the sheaf on
\[\Hilb_\beta^{n}(S) \times S = S^{[n_0]} \times \cdots\times S^{[n_s]} \times |\O_S(\beta_1)|\times\cdots\times |\O_S(\beta_s)| \times S\,\]
and for its restriction to the nested Hilbert scheme
\[i:S_\beta^{[n]} \hookrightarrow \Hilb_\beta^{n}(S)\,,\]
over which we have a canonically defined Higgs field
$\phi_\L \colon E_\L^{[n]}\rightarrow E_\L^{[n]} \otimes \omega_S\otimes \t$.

%

Define a class
\begin{align*}
T_\L^{[n]}
	\coloneqq {} & \RHom_\pi(E_\L^{[n]},E_\L^{[n]}\otimes\omega_S\otimes\t)_0 - \RHom_\pi(E_\L^{[n]},E_\L^{[n]})_0 \\
	\in {} & K_0^{\CC^*}(\Hilb_\beta^{n}(S))\,,
\end{align*}
and denote its pull-back to $S_\beta^{[n]}$ by the same symbol. Note that $T_\L^{[n]}$ depends only on $\beta$, rather than on $L$ (or on $\L$). We will write
\[N_\L^{[n]} \coloneqq T_\L^{[n]} - \left(T_\L^{[n]}\right)^{\CC^*}\]
for its moving part.
Let $e$ denote the $\CC^*$-equivariant Euler class, and define the rational number
\begin{equation}\label{EqContVW}
\VW_\beta^{[n]} \coloneqq \int_{[S_\beta^{[n]}]^{\vir}} \frac{1}{e\left(N_\L^{[n]}\right)}\,. 
\end{equation}
In the case that $(E_\L^{[n]},\phi_\L)$ represents a connected component
\[\M_L^{[n]} = (S_\beta^{[n]},E_\L^{[n]},\phi_\L^{[n]})\subset \M_{1^r,c_1,c_2}\,,\]
$T_\L^{[n]}$ is the class in $K$-theory of the cone \eqref{EqCone} in the introduction, and hence equals the $\CC^*$-localized perfect obstruction theory of \cite{TT} on $\M_L^{[n]}$. Over $\M_L^{[n]}$, the class $N_\L^{[n]}$ is the virtual normal bundle to the $\CC^*$ fixed locus $ (\N_{r,M,c_2}^\bot )^{\CC^*}$ in $\N_{r,M,c_2}^\bot$. By definition of the Vafa-Witten invariant \eqref{EqDefVW}, the contribution of the connected component $\M_L^{[n]}$ is given by $\VW_\beta^{[n]}$.

If the Higgs pair $(E_\L^{[n]},\phi_\L)$ contains fibres that are not Gieseker semistable, it does not represent a connected component of any $\M_{1^r,c_1,c_2}$, and hence does not contribute to the Vafa-Witten invariant. On the other hand, by Proposition \ref{PropGSs}, we have $\VW_\beta^{[n]}= 0$ in this case. It follows that, using the notation from Proposition \ref{PropVirClass}, we have:
\begin{align*}
\VW_{1^r,c_1,c_2}
	& = \sum_{L,n} \VW_{\beta(L)}^{[n]} \\
	& = \sum_{\beta,n} \VW_{\beta}^{[n]} \cdot  \#\ker \left(H^2(S,\ZZ) \xrightarrow{\cdot (s+1)} H^2(S,\ZZ)\right) \,.
\end{align*}

Now define a line bundle
\[K_\L^{[n]}  \coloneqq \det \left(T_\L^{[n]\vee}\right)\]
on $\Hilb_\beta^{[n]}(S)$. Note that $T^{[n]}_{\beta}$ is defined as the difference between a complex and its dual, up to a factor $\t$. Hence its determinant is by construction a square, up to a factor $\t$. Hence, after choosing once and for all a square root of $\t$, the line bundle $K^{[n]}_\beta$ has a canonical square root, denoted by $(K_\L^{[n]})^\frac{1}{2}$. Over $S_\beta^{[n]}$, the bundle $K_\L^{[n]}$ restricts to the virtual canonical bundle \cite{T}, and its square root restricts to the canonical square root of \cite[Proposition 2.6]{T}.

By \cite{T}, the contribution to the refined invariant can be computed by
\begin{equation}\label{EqContVWref}
\VW_\beta^{[n]}(y) \coloneqq\left[\int_{[S_\beta^{[n]}]^{\vir}} \frac{\ch\left((K_\L^{[n]})^\frac{1}{2}\right)}{\ch\left(\Lambda^\bullet(N_\L^{[n]\vee})\right)}\Td\left((T_\L^{[n]})^{\CC^*}\right)\right]_{\ch(\t) = y}\,,
\end{equation}
where $\ch$ and $\Td$ denote the $\CC^*$-equivariant Chern character and Todd class respectively. Again, in the language of Proposition \ref{PropVirClass}, we have
\begin{equation}\label{EqSumContRef}
\begin{split}
\VW_{1^r,c_1,c_2}(y)
	& = \sum_{L,n} \VW_{\beta(L)}^{[n]}(y) \\
	& = \sum_{\beta,n} \VW_{\beta}^{[n]}(y) \cdot \#\ker \left(H^2(S,\ZZ) \xrightarrow{\cdot (s+1)} H^2(S,\ZZ)\right) \,.
\end{split}
\end{equation}


By Theorem \ref{TheComp}, we have
\begin{multline*}
i_*[S_\beta^{[n]}]^{\vir} = \prod_{i=1}^{s} e\left(R\pi_*(\O_S(\beta_i) - \RHom_\pi\left(\I^{[n_{i-1}]},\I^{[n_i]}(\beta_i)\right)\right) \\
\cap \left[S^{[n_0]} \times \cdots \times S^{[n_{s}]}\right] \times \SW(\beta_1)\times\cdots\times\SW(\beta_s) \,.
\end{multline*}
The factor
\begin{align*}
\SW(\beta)
	\coloneqq {} & \SW(\beta_1)\times\cdots\times\SW(\beta_s) \\
	\in {} & A_0\big(|\O_S(\beta_1)|\times\cdots\times|\O_S(\beta_s)|\big)
\end{align*}
annihilates all Chern classes in the integrants of \eqref{EqContVW} and \eqref{EqContVWref} that are pulled back from
\[|\O_S(\beta_1)|\times\cdots\times|\O_S(\beta_s)|\,.\]
It follows that we can rewrite \eqref{EqContVW} and \eqref{EqContVWref} as integrals over
\[\Hilb^n(S) = S^{[n_0]}\times \cdots\times S^{[n_s]}\,.\]

Define the sheaf
\[E_L^{[n]} \coloneqq L_0 \otimes \I^{[n_0]} \otimes \t^0 \oplus \ldots,\oplus L_s \otimes \I^{[n_s]} \otimes \t^{-s} \quad \text{on}\quad \Hilb^n(S)\times S\]
and classes
\begin{align*}
T_L^{[n]} &\coloneqq \RHom_\pi(E_L^{[n]},E_L^{[n]}\otimes\omega_S\otimes\t)_0 - \RHom_\pi(E_L^{[n]},E_L^{[n]})_0\,,\\
N_L^{[n]} &\coloneqq T_L^{[n]} - \left(T_L^{[n]}\right)^{\CC^*}\,, \quad \text{and}\\
K_L^{[n]} &\coloneqq \det \left(T_L^{[n]\vee}\right)
\end{align*}
in $K_0(\Hilb^{n}(S))$. Again, note that since $H^1(\O_S) = 0$, the classes $T_L^{[n]}$, $N_L^{[n]}$, and $K_L^{[n]}$ depend on $\beta = \beta(L)$, rather than on $L$. We have, now considering $\SW(\beta)$ as an integer,
\begin{multline}\label{EqVW}
VW_\beta^{[n]}=\SW(\beta) \int_{[\Hilb^{n}(S)]} \frac{1}{e\left(N_L^{[n]}\right)} \\
\times \prod_{i=1}^s e\left(R\pi_*\O(\beta_i) - \RHom_\pi(\I^{[n_{i-1}]},\I^{[n_i]}(\beta_i))\right)
\end{multline}
and
\begin{multline}\label{EqVWref}
VW_\beta^{[n]}(y)=\SW(\beta) \Bigg[\int_{[\Hilb^{n}(S)]} \frac{\ch\left((K_L^{[n]})^\frac{1}{2}\right)}{\ch(\Lambda^\bullet(N_L^{[n]\vee}))}\Td\left((T_L^{[n]})^{\CC^*}\right)\\
\times  \prod_{i=1}^s e\left(R\pi_*\O(\beta_i) - \RHom_\pi(\I^{[n_{i-1}]},\I^{[n_i]}(\beta_i))\right)\Bigg]_{\ch(\t) = y}\,.
\end{multline}

\section{Removing trace} \label{SecTrace}
We can normalize the generating series
\[
\sum_n \VW_\beta^{[n]} \, q^{n}
\]
by dividing through the leading term. In terms of the integrals of \eqref{EqVW}, this comes down to considering `traceless' integrants. By this we mean the following. Note that $N_L$ can be written as a linear combination of terms of the form
\[\RHom_\pi(\mathcal{E}, \mathcal{F})\]
with $\mathcal{E}$ and $\mathcal{F}$ torsion free rank $1$ sheaves. We will replace each such term by
\[\RHom_\pi(\mathcal{E},\mathcal{F}) - \RHom_\pi(\det\mathcal{E},\det\mathcal{F})\,.\]
In Section \ref{SecFront}, we will deal with the leading term of the generating series separately.

We keep the notation from the previous section. Moreover, we will write
\begin{align*}
E_L &\coloneqq E_L^{[0]} = L_0\otimes \t^0 \oplus \ldots \oplus L_s \otimes \t^{-s}
\end{align*}
for the vector bundle on $S$, and furthermore
\begin{align*}
T_L &\coloneqq T_L^{[0]}=  R\Hom(E_L,E_L\otimes \omega_S\otimes \t) _0- R\Hom( E_L,E_L)_0\,,\\
N_L & \coloneqq N_L^{[0]}=T_L - (T_L)^{\CC^*}\,,\\
K_L & \coloneqq K_L^{[0]}= \det T_L^\vee\,.
\end{align*}
for the classes in the equivariant $K$-group of a point.
Finally, we will also use the notation
\[T_{L,0}^{[n]} = T_L^{[n]} - T_L\,, \quad N_{L,0}^{[n]} = N_L^{[n]} - N_L\,, \quad \text{and} \quad K_{L,0}^{[n]} = K_L^{[n]} \otimes K_L^*\,,\]
for the classes in $K_0(\Hilb^n(S))$, where we suppress pull-backs from the point. Define
\begin{align*}
F_n(S,\beta)
	& \coloneqq \frac{1}{e(N_L)} \quad \text{and}\\
Q_n(S,\beta)
	& \coloneqq \int_{[\Hilb^{n}(S)]} \frac{1}{e\left(N_{L,0}^{[n]}\right)} \prod_{i=1}^s e\left(R\pi_*\O(\beta_i) - \RHom_\pi(\I^{[n_{i-1}]},\I^{[n_i]}(\beta_i))\right)\,,
\end{align*}
so we have
\[VW_\beta^{[n]} = \SW(\beta) F_n(S,\beta) Q_n(S,\beta)\,.\]
In the refined case, define
\begin{align*}
F(S,\beta,y)
	\coloneqq{}& \left[\frac{\ch\left(K_L^{\frac{1}{2}}\right)}{\ch(\Lambda^\bullet(N_L^\vee))}\Td\left(T_L^{\CC^*}\right)\right]_{\ch(\t) = y} \quad \text{and} \\
Q_n(S,\beta,y)
	\coloneqq{}& \Bigg[\int_{[\Hilb^{n}(S)]} \frac{\ch\left((K_{L,0}^{[n]})^{\frac{1}{2}}\right)}{\ch(\Lambda^\bullet((N_{L,0}^{[n]})^\vee))}\Td\left((T_{L,0}^{[n]})^{\CC^*}\right)\\
	& \times  \prod_{i=1}^s e\left(R\pi_*\O(\beta_i) - \RHom_\pi(\I^{[n_{i-1}]},\I^{[n_i]}(\beta_i))\right)\Bigg]_{\ch(\t) = y}\,,
\end{align*}
so that
\begin{equation}\label{EqVWContFac}
\VW_\beta^{[n]}(y) = \SW(\beta) \, F(S,\beta,y) \, Q_n(S,\beta,y)\,.
\end{equation}

\begin{remark}\label{RemExpY}
A priori, $Q_n(S,\beta,y)$ is a rational function in $\sqrt y$, due to the fractional exponent of the virtual canonical bundle. However, an easy computation shows that the equivariant parameter $\t$ appears in $K^{[n]}_{\beta,0}$, with \emph{even} exponent, and hence, $Q_n(S,\beta,y)$ is in fact a rational function in $y$.
\end{remark}

In the Section \ref{SecFront} we will compute $F(S,\beta)$ under the assumption $\SW(\beta)\neq 0 $. In Section \ref{SecUniv} we will show that the numbers $Q_n(S,\beta)$ are given by universal polynomials $P_n(S,\beta)$ in the Chern numbers of $S$ and $\beta=(\beta_1,\ldots,\beta_s)$. We will deal with the refined version at the same time.

\section{The leading term}\label{SecFront}
In this section we compute the factor $F(S,\beta,y)$. Let $L=(L_0,\ldots,L_s)$ be an $(s+1)$-tuple of line bundles on $S$, and let $\beta=\beta(L)$ be given as in Notation~\ref{Notation}.
Also recall that we write
\[E_L = L_0 \otimes \t^0 \oplus \ldots \oplus L_s \otimes \t^{-s}\,.\]
Assume that
\[\SW(\beta) = \SW(\beta_1)\cdots\SW(\beta_s) \neq 0 \,.\]
Then, by \cite[Proposition 6.29]{Mo}, we have $(\beta^i)^2 = (\beta^i K_S)$, or equivalently $\chi(\beta^i) = \chi(\O_S)$ for $i=1,\ldots,s$. Using Serre duality, we can write
\begin{align*}
T_L	= {}& R\Hom(E_L,E_L\otimes \omega_s\otimes \t) _0- R\Hom( E_L,E_L)_0\\
	= {}& \sum_{i=0}^s (\t^{i+1} - \t^{-i}) \left(\sum_{j=0}^{s-i}\chi(L_{i+j}^* \otimes L_j \otimes \omega_S)
			- \sum_{j=1}^{s-i}\chi(L_{i+j}^*\otimes L_{j-1})\right) \\
	& - (\t - 1)\cdot\chi(\O_s)\\
	= {}& \sum_{i=1}^s (\t^{i+1} - \t^{-i}) \left(\sum_{j=0}^{s-i}\chi(\beta^{1+j}+\ldots + \beta^{i+j}+K_S)
			- \sum_{j=1}^{s-i}\chi(\beta^{j}+\ldots+\beta^{i+j})\right)\,,\\
\end{align*}
where the second sum starts with $i=1$, since the coefficient of $(\t - 1)$ equals
\[\left(\sum_{j=0}^{s}\chi(K_S) - \sum_{j=1}^{s}\chi(\beta^{j})\right) - \chi(\O_S) = 0\]
by the assumption $\SW(\beta) \neq 0$. Note that in particular, we have
\begin{equation}\label{EqNoFix}
T_L = N_L\,.
\end{equation}
Moreover, note that
\begin{align*}
\chi(\beta^{1+j}+\ldots + \beta^{i+j}+K_S)
	={}&\frac{(\beta^{1+j}+\ldots + \beta^{i+j}+K_S)\cdot(\beta^{1+j}+\ldots + \beta^{i+j})}{2}\\
	&+ \chi(\O_S)\\
	={}&\sum_{j<k\leq l\leq i+j} \beta^k\beta^l + \chi(\O_S)\,.
\end{align*}
and similarly
\begin{align*}
\chi(\beta^{j}+\ldots + \beta^{i+j})
	&=\frac{(\beta^{j}+\ldots + \beta^{i+j})\cdot(\beta^{j}+\ldots + \beta^{i+j}-K_S)}{2} + \chi(\O_S)\\
	&=\sum_{j\leq k< l\leq i+j} \beta^k\beta^l + \chi(\O_S)\,.
\end{align*}
It follows that for $k\leq l$, the multiplicity with which the term
\[(\t^{i+1} - \t^{-i})\cdot \beta^k\beta^l\]
appears in $T_L$ is given by
\begin{align*}
\mu(i,k,l)
	\coloneqq {} & \#\{j \mid 0 \leq j < k \leq l \leq i+j \leq s\}- \#\{j \mid 0 < j \leq k < l \leq i+j \leq s\}\\
	= {} & \#\{j \mid 0,l-i \leq j \leq k - 1, s - i\}- \#\{j \mid 1,l-i \leq j \leq k, s - i; k<l\}\\
	= {} &
\begin{cases}\left.
	\begin{cases}
	-1&\mbox{if } l-k \leq i< \min (l, s - k +1)\\
	1&\mbox{if } \max(l, s - k + 1) \leq i\leq s\\
	0&\mbox{else}
	\end{cases} \right\}
		&\mbox{if } k<l\\
	\min\{i,s-k+1,k,s-i+1\}
		&\mbox{if } k=l \,,
\end{cases}
\end{align*}
so we have
\begin{align}\label{EqFrontFF}
N_L = T_L &= \sum_{i=1}^s \left(\chi(\O_S)+ \sum_{k\leq l} \mu(i,k,l)\cdot \beta^k\beta^l \right)\cdot(\t^{i+1} - t^{-i})\,.
\end{align}

We define the following rational numbers:
\begin{alignat*}{3}
F^{(s+1)}_0
	&\coloneqq \frac{(-1)^{s}}{s+1}\,;\\
F^{(s+1)}_{kk} 
	&\coloneqq \frac{(-1)^{sk}} {\binom{s+1}{k}} \quad
	&&\text{for}\quad 1\leq k\leq s \,;\\
F^{(s+1)}_{kl}
	&\coloneqq \frac{l(s+1-k)}{(l-k)(s+1)} \quad
	&&\text{for}\quad 1\leq k<l \leq s \,.\\
\end{alignat*}

\begin{proposition}\label{PropFrontUnref}
Recall that we assume $\SW(\beta) \neq 0$. We have
\[F(S,\beta)
=\left(F^{(s+1)}_0\right)^{\chi(\O_S)}
\prod_{k\leq l} \left(F^{(s+1)}_{kl}\right)^{\beta^k\beta^l}\,.
\]
\end{proposition}

\begin{proof}
Applying $\frac{1}{e(\,)}$ to equation \eqref{EqFrontFF}, we obtain
\begin{align*}
F(S,\beta)
	& = \frac{1}{e(N_L)} \\
	& = e\left( \sum_{i=1}^s  \t^{i+1} - t^{-i} \right)^{-\chi(\O_S)}
	\prod_{k\leq l} e\left(\sum_{i=1}^s \mu(i,k,l) (\t^{i+1} - t^{-i}) \right)^{-\beta^k\beta^l}
\end{align*}
Note that we have
\[\frac{1}{e(\t^{i+1} - \t^{-i})}= \frac{-i}{i+1}\,,\]
and hence
\[e\left(\sum_{i=1}^s (\t^{i+1} - \t^i)\right)^{-1} = \frac{-1 \cdots -i}{2 \cdots (i+1)} = \frac{(-1)^s}{s+1}\]
For $k< l$ we find
\begin{align*}
e\left(\sum_{i=1}^s \mu(i,k,l) \cdot(\t^{i+1} - \t^{-i})\right)^{-1}
	& = \prod_{i = l-k}^{\min(l-1,s-k)} \left(\frac{-i}{i+1}\right)^{-1} \cdot \prod_{i=\max(l,s-k+1)}^s \frac{-i}{i+1}\\
	& = \frac{l( s+1-k)}{(l-k)(s+1)}\,.
\end{align*}
Finally, write
\[a \coloneqq \min(k,s+1-k) \quad \text{and}\quad b\coloneqq \max(k,s+1-k)\,,\]
so we have
\begin{align*}
e\left(\sum_{i=1}^s \mu(i,k,k) \cdot(\t^{i+1} - \t^{-i})\right)^{-1} 
	& =
		\prod_{i = 1}^{a}	\left(\frac{-i}{i+1}\right)^i
		\prod_{i=a+1}^{b-1}	\left(\frac{-i}{i+1}\right)^a
		\prod_{i=b}^s		\left(\frac{-i}{i+1}\right)^{s+1-i} \\
	& = (-1)^{sk} \frac{1\cdots a}{(b+1)\cdots (s+1)} \\
	& =  \frac{(-1)^{sk}}{\binom{s+1}{k}} \qedhere
\end{align*}
\end{proof}

\begin{notation}
We will use \emph{quantum integers}, which are given by
\[ [i]_y \coloneqq \frac{y^{i/2} - y^{-i/2}}{y^{1/2} - y^{-1/2}}\,.\]
We will also use the notation
\[ \binom{i}{j}_y \coloneqq \frac{ [i]_y \cdots  [i - j +1]_y}{[1]_y \cdots [j]_y} \]
for non-negative integers $i\leq j$.
\end{notation}

Define the following rational functions in $y^{1/2}$:

\begin{alignat*}{3}
F^{(s+1)}_0(y)
	&\coloneqq \frac{(-1)^s}{[s+1]_y}\,;\\
F^{(s+1)}_{kk}(y)
	&\coloneqq \frac{(-1)^{sk}}{\binom{s+1}{k}_y} \quad
	&& \text{for}\quad 1\leq k\leq s \,;\\
F^{(s+1)}_{kl}(y)
	&\coloneqq \frac{[l]_y [s+1-k]_y}{[l-k]_y [s+1]_y } \quad
	&& \text{for}\quad 1\leq k<l \leq s \,.\\
\end{alignat*}

\begin{proposition}\label{PropFront}
Assume that $\SW(\beta) \neq 0$. Then we have
\[F(S,\beta,y)
=\left(F^{(s+1)}_0(y)\right)^{\chi(\O_S)}
\prod_{k\leq l} \left(F^{(s+1)}_{kl}(y)\right)^{\beta^k\beta^l}\,.
\]
\end{proposition}
\begin{proof}
Recall \eqref {EqNoFix} that $T_L$ has no fixed part, so we have
\begin{align*}
F(S,\beta,y)
	& = \left[\frac{\ch(K_L^{\frac{1}{2}})}{\ch(\Lambda^\bullet(N_L^\vee))} \Td((T_L)^{\CC^*}) \right]_{\ch(\t) = y}\\
	& = \left[\frac{\ch(\det(N_L^\vee))^{\frac{1}{2}}}{\ch(\Lambda^\bullet(N_L^\vee))} \right]_{\ch(\t) = y}\,.
\end{align*}
Note that we have
\begin{align*}
\left[\frac{\ch(\det((\t^{i+1} - \t^{-i})^\vee)^{\frac{1}{2}})}{\ch(\Lambda^\bullet((\t^{i+1} - \t^{-i})^\vee))} \right]_{\ch(\t) = y}
	& = -y^{1/2}\frac{y^i-1}{y^{i+1}-1} \\
	& = - \frac{[i]_y}{[i+1]_y} \,.
\end{align*}
Now follow the proof of Proposition~\ref{PropFrontUnref}.

\end{proof}

\begin{remark}\label{RemExpYII}
If $r = s+1$ is odd, note that $F(S,\beta,y)$ is a function in $y$, rather than in $\sqrt y$, for any $\beta = (\beta_0,\ldots,\beta_s)$ with $\SW(\beta)\neq 0$.
\end{remark}

\begin{example}
For rank $2$ we have
\[F(S,\beta,y) = \left(\frac{-y^{1/2}}{1+y}\right)^{\chi(\O_S)+\beta^1\beta^1}\,,\]
and for rank $3$
\[F(S,\beta,y) = \left(\frac{y}{1+y+y^2}\right)^{\chi(\O_S)+\beta^1\beta^1+\beta^2\beta^2}\left(\frac{(y+1)^2}{1+y+y^2}\right)^{\beta^1\beta^2}\,.\]
\end{example}

\section{Universality}\label{SecUniv}
Let $S$ be a smooth projective surface, not necessarily with $H^1(\O_S)=0$ or $p_q>0$. For non-negative integers $n=(n_0,\ldots,n_s)$ and classes $\beta = (\beta_1,\ldots,\beta_s)$, consider the rational number $Q_n(S,\beta)$ defined in Section~\ref{SecTrace} as an integral over
\[\Hilb^n(S) = S^{[n_0]}\times \cdots \times S^{[n_s]}\,.\]
Using the notation
\[q^n \coloneqq q_0^{n_0}\cdots q_s^{n_s}\]
we form the generating series
\[\sum_n Q_n(S,\beta)\, q^n\,.\]
The following universality result Proposition \ref{PropUnivUnref}, or rather its refinement Proposition \ref{PropUniv}, is the main ingredient for the proof of Theorem \ref{Theorem1}.

\begin{remark}\label{RemDivCl}
In Section \ref{SecTrace}, the integrals $Q_n(S,\beta)$ were defined in terms of a lift of $\beta$ to a vector of line bundles $L$, such that $\beta = \beta(L)$ (see Notation \ref{Notation}). Since we do not assume $H^1(\O_S)=0$ in this section, this lift involves a a lift of the $\beta_i$ to divisor classes of $S$. We assume we have made such choice, and we consider $\beta$ as a vector of classes in $A^1(S)$. By Proposition~\ref{PropUnivUnref}, $Q_n(S,\beta)$ does not depend on this choice.
\end{remark}

\begin{proposition}\label{PropUnivUnref}
For each symbol
\[\underline{\mathfrak{N}}\in \left\{\underline{\chi(\O_S)},\underline{K_S^2},\underline{K_S\beta^i}, \underline{\beta^i \beta^j}\right\}_{1\leq i\leq j\leq s}\]
there is a power series $A^{(s+1)}_{\underline{\mathfrak{N}}} \in \QQ[[q_0,\ldots,q_s]]$, starting with $1$, and depending only on $s$, such that
\[\sum_n Q_n(S,\beta)\, q^n= \prod_{\underline{\fN}}(A^{(s+1)}_{\underline{\mathfrak{N}}})^{\fN}\]
for \emph{any} smooth projective surface $S$ and classes $\beta_1,\ldots,\beta_s\in A^1(S)$.
\end{proposition}

\begin{proof}
By the techniques of \cite{EGL} (see also \cite{GNY}), the integral $Q_n(S,\beta)$ can be universally expressed as a polynomial $P_n(S,\beta)$ in the Chern numbers of $S$ and the classes $\beta^1,\ldots,\beta^s$. Following \cite[Proposition 2.3]{Goe}, it suffices to show that the generating series is multiplicative, i.e., that we have\
\[\sum Q_n(S\sqcup S',\beta+\beta')\,q^n = \sum Q_n(S,\beta)\,q^n \cdot\sum Q_n(S',\beta')\,q^n\]
for surfaces $S$ and $S'$ and $s$-tuples $\beta$ and $\beta'$ of classes in $A^1(S)$ and $A^1(S')$ respectively.

Note that
\begin{align}
\Hilb^n (S\sqcup S')	&= (S\sqcup S')^{[n_0]}\times\cdots\times(S\sqcup S')^{[n_s]} \nonumber \\
				&= \bigsqcup_{i_0+j_0 = n_0} (S^{[i_0]}\times S'^{[j_0]})\times\cdots\times\bigsqcup_{i_s+j_s = n_s} (S^{[i_s]}\times S'^{[j_s]}) \nonumber\\
				&= \bigsqcup_{\substack{i_0+j_0=n_0,\\\cdots\\i_s+j_s=n_s}} S^{[i_0]}\times S'^{[j_0]}\times\cdots\times S^{[i_s]}\times S'^{[j_s]} \nonumber\\
				&= \bigsqcup_{i+j = n} \Hilb^i(S) \times \Hilb^j(S')\,, \label{EqDec}
\end{align}
in which the last sum is taken over $s+1$-tuples $i = (i_0,\ldots,i_s)$ and $j= (j_0,\ldots,j_s)$ of non-negative integers with $n = i + j$. Consider the universal ideal sheaves
\[\I_{S\sqcup S'}^{[n_k]} \quad \mathrm{for} \quad k=0,\ldots,s\]
on
\[\Hilb^n(S\sqcup S')\times (S\sqcup S')\,.\]
For fixed $i$ and $j$ with $i+j = n$ and for $k = 0,\ldots,s$, we will write
\begin{align*}
p_k
	&\colon \Hilb^i(S)\times \Hilb^j(S') \times S
	 \rightarrow S^{[i_k]}\times S\\
q_k
	&\colon \Hilb^i(S)\times \Hilb^j(S') \times S'
	 \rightarrow S'^{[j_k]}\times S'
\end{align*}
for the projections. Over the components in the decomposition~\eqref{EqDec}, the universal sheaves are given by
\begin{align*}
\I_{S\sqcup S'}^{[n_k]}\Big\vert_{\Hilb^i(S)\times \Hilb^j(S') \times (S\sqcup S')}
	 = p_k^*\I^{[i_k]}_S  \oplus q_k^* \I_{S'}^{[j_k]}\,.
\end{align*}
Write
\[\pi\colon S\rightarrow *\,, \quad \pi'\colon S'\rightarrow *\quad \mathrm{and} \quad \pi\sqcup \pi'\colon S\sqcup S'\rightarrow *\]
for the projections.
Let $M$ and $M'$ be a line bundles on $S$ and $S'$ respectively. It follows that for $0\leq k,l \leq s$ we have
\begin{multline} \label{EqRHom}
\RHom_{\pi\sqcup\pi'}(\I_{S\sqcup S'}^{[n_k]},\I_{S\sqcup S'}^{[n_l]}\otimes (M\oplus M')) = \\
\sum_{i+j=n} \RHom_\pi(p_k^*\I_{S}^{[i_k]},p_l^*\I_{S}^{[i_l]}\otimes M)\oplus \RHom_{\pi'}(q_k^*\I_{S'}^{[j_k]},q_l^*\I_{S'}^{[j_l]}\otimes M')
\end{multline}
in the ring
\[K_0(\Hilb^n (S\sqcup S')) = \bigoplus_{i+j = n} K_0(\Hilb^i(S)\times\Hilb^j(S'))\,.\]
For any pair $i,j$ of $(s+1)$-tuples of non-negative integers, write
\begin{align*}
p
	& \colon \Hilb^i(S)\times\Hilb^j(S')\rightarrow  \Hilb^i(S)\\
q
	& \colon \Hilb^i(S)\times\Hilb^j(S')\rightarrow  \Hilb^j(S')\,.
\end{align*}

Let $L$ and $L'$ be $s+1$-tuples  of line bundles on $S$ and $S'$ respectively, such that $\beta = \beta(L)$ and $\beta' = \beta(L')$ (see notation \ref{Notation}). Consider the $K$-theory classes
\[N_{L\oplus L',0}^{[n]}\,,\quad N_{L,0}^{[i]}\quad \text{and}\quad N_{L',0}^{[j]}\]
as defined in Section \ref{SecVWInt}. Note that it is immediate from the definition, that these classes do not depend on the choice of $L$ and $L'$ (see also Remark \ref{RemDivCl}). By definition, $N^{[n]}_{L+L',0}$ is linear combination of classes of the form \eqref{EqRHom}, and we find
\[
N^{[n]}_{L+L',0} =
\sum_{i+j = n} p^*N^{[i]}_{L,0} + q^*N^{[j]}_{L',0}\,\,.
\]
It follows that
\[\frac{1}{e(N^{[n]}_{L+L',0})} = \sum_{i+j=n} \frac{1}{e(N^{[i]}_{L,0})}\cdot \frac{1}{e(N^{[n]}_{L',0})}\,.\]
Finally, the corresponding multiplicative property of the factor
\[\prod_{i=k}^s e\left(R\pi_*\O(\beta_k) - \RHom_\pi(\I^{[n_{k-1}]},\I^{[n_k]}(\beta_k))\right)\]
in the integrant of $Q_n(S,\beta)$
follows from the generalized Carlsson-Okounkov vanishing of \cite{GT} (see Remark~\ref{RemOkVan}). Integrating gives the result.
\end{proof}


The proof of Proposition \ref{PropUnivUnref} also gives the following refined result.

\begin{proposition} \label{PropUniv}
For each symbol
\[\underline{\mathfrak{N}}\in \left\{\underline{\chi(\O_S)}, \underline{K_S^2},\underline{K_S\beta^i}, \underline{\beta^i \beta^j}\right\}_{1\leq i\leq j\leq s}\]
there is a power series $A^{(s+1)}_{\underline{\mathfrak{N}}}(y) \in \QQ(y)[[q_0,\ldots,q_s]]$, starting with $1$, such that
\[\sum_n Q_n(S,\beta,y)\, q^n= \prod_{\underline{\fN}}(A^{(s+1)}_{\underline{\mathfrak{N}}}(y))^{\fN} \]
for \emph{any} smooth projective surface $S$ and classes $\beta_1,\ldots,\beta_s\in A^1(S)$.
\end{proposition}
\begin{proof}
A similar proof holds, using the multiplicative properties of $\ch$, $\Lambda^\bullet$, $\det$, and $\Td$.
Note that by Remark \ref{RemExpY}, the universal series take coefficients in $\QQ(y)$, rather than in $\QQ(\sqrt y)$.
\end{proof}


\section{Proof of Theorem A} \label{SecProof}
As usual, we fix a rank $r = s+1$. In this section, we will identify
\[q \coloneqq q_0=\ldots = q_{s}\,,\]
so the equation in Proposition \ref{PropUniv} becomes
\begin{equation}\label{EqProofQ}
\sum_n Q_n(S,\beta,y)\, q^{|n|}= \prod_{\underline{\fN}}(A^{(r)}_{\underline{\mathfrak{N}}}(y))^{\fN}
\end{equation}
in the ring $\QQ(y)[[q]]$, where we use the notation $|n| = n_0+\ldots+n_s$.

Let $L = (L_0,\ldots,L_s)$ be line bundles on a surface $S$, and let
\begin{align*}
\beta = \beta(L)
	& = (\beta_1,\ldots,\beta_s) \\
	& = (K_S - \beta^1, \ldots, K_S - \beta^s)
\end{align*}
be given as in Notation \ref{Notation}. For non-negative integers $n=(n_0,\ldots,n_s)$ and ideal sheaves $I_i\in S^{[n_i]}$, consider the sheaf
\[E = L_0\otimes I_0 \oplus \ldots \oplus L_s\otimes I_s\,.\]
In the notation of Section \ref{SecMod}, $E$ is a fibre of the family $E_\L^{[n]}$ of sheaves on $S$ over $\Hilb_\beta^{n}(S)$. By Lemma \ref{LemmaChern} we have
\[c_2(E) = |n| +  \frac{r-1}{2r} c_1(E)^2 - \sum_{i<j}  \frac{i (r - j)}{r} \beta^i\beta^j  - \sum_i \frac{i (r - i)}{2r}  (\beta^i)^2\,.\]
We will write
\[d(\beta) \coloneqq - \sum_{i<j}  \frac{i (r - j)}{r} \beta^i\beta^j  - \sum_i \frac{i (r - i)}{2r}  (\beta^i)^2\,,\]
so that we have
\begin{equation}\label{Eqc2Degree}
q^{\frac{1-r}{2r} c_1(E)^2} q^{c_2(E)} = q^{|n| + d(\beta)} \,.
\end{equation}
Finally, recall that for any surface $S$ with $H^1(\O_S)=0$ and $p_g(S)>0$, and for $\beta$ with
\[\SW(\beta) = \SW(\beta_1)\cdots\SW(\beta_s) \neq 0\]
we have, by Proposition \ref{PropFront},
\begin{equation}\label{EqProofF}
F(S,\beta,y)
=\left(F^{(r)}_0(y)\right)^{\chi(\O_S)}
\prod_{k\leq l} \left(F^{(r)}_{kl}(y)\right)^{\beta^k\beta^l}\,,
\end{equation}
where $F(S,\beta,y)$ is defined as in Section~\ref{SecTrace}.

Define the following Laurent series in $q^\frac{1}{2r}$ with coefficients in $\QQ(\sqrt y)$:
\begin{align*}
A
	& \coloneqq F^{(r)}_0(y)\,
		A^{(r)}_{\underline{\chi(\O_S)}}(y)\, (-1)^{(r-1)} \\
B
	& \coloneqq A^{(r)}_{\underline{K_S^2}}(y) \\
C_{ij}
	& \coloneqq q^{\frac{i(j-r)}{r}}\, F^{(r)}_{ij}(y)\,
		A^{(r)}_{\underline{\beta^i\beta^j}}(y)
		\quad \text{for}\quad 1\leq i<j \leq r-1\,;\\
C_{ii}
	& \coloneqq q^{\frac{i(i-r)}{2r}}\, F^{(r)}_{ii}(y)\,
		A^{(r)}_{\underline{\beta^i\beta^i}}(y)\, A^{(r)}_{\underline{\beta^iK_S}}(y)
		\quad \text{for} \quad 1 \leq j \leq r-1 \,.
\end{align*}
\begin{proof}[Proof of Theorem \ref{Theorem1}]
First note that, by definition, the Laurent series are universal in the sense that they \emph{only} depend on $r$. Now let $S$ be a surface with $H^1(\O_S)=0$ and $p_g(S)>0$. We have
\begin{align*}
\Z_{S,r,c_1}(q,y)
	& = \frac{q^{\frac{1-r}{2r} c_1^2}}{\#H^2(S,\ZZ)[r] } \, \sum_{c_2\in\ZZ} \VW_{1^r,c_1,c_2}(S,y)\, q^{c_2} \\
	& = q^{\frac{1-r}{2r} c_1^2} \, \sum_{c_2\in\ZZ} \hat{\sum_{\beta,n}} \VW_\beta^{[n]}(y)\, q^{c_2}
		& \text{by \eqref{EqSumContRef}} \\
	& = \sum_\beta \sum_{n\in (\ZZ_{\geq 0})^r} \VW_\beta^{[n]}(y)\, q^{|n| + d(\beta)}
		& \text{by \eqref{Eqc2Degree}} \\
	& = \sum_\beta \SW(\beta)F(S,\beta,y) \sum_{n\in (\ZZ_{\geq 0})^r} Q_n(S,\beta,y)\, q^{|n| + d(\beta)}
		& \text{by \eqref{EqVWContFac}} \\
	& = \sum_\beta \SW(\beta^\vee) A^{\chi(\O_S)} B^{K_S^2} \prod_{i\leq j} C_{ij}^{\beta^i\beta^j} \,.
		& \text{by \eqref{EqProofQ} and \eqref{EqProofF}}
\end{align*}
Here the symbol {\Small $\displaystyle\sum_\beta$} denotes a sum over $(r-1)$-tuples $\beta = (\beta_1,\ldots,\beta_{r-1})$ satisfying
\[c_1 \equiv \sum_{i=1}^{r-1} i \beta^i \mod rH^2(S,\ZZ)\,.\]
The sum {\Small $\displaystyle\hat{\sum_{\beta,n}}$}
is taken over $\beta$ as above, and $n \in (\ZZ_{\geq 0})^r$ with
\[c_2 = |n| + \frac{r-1}{2r} c_1^2 + d(\beta)\]
cf.\ Proposition~\ref{PropVirClass}. Finally, we have used the notation $\beta^\vee = (\beta^1,\ldots,\beta^{r-1})$, and the equation
\begin{align*}
\SW(\beta^\vee)
	& = \SW(\beta^1)\cdots \SW(\beta^{r-1}) \\
	& = (-1)^{(r-1)\chi(\O_S)} \, \SW(\beta) \,,
\end{align*}
which follows from \cite[Proposition 6.3.4]{Mo}.
\end{proof}

\begin{proposition}\label{PropSqrt}
Let $S$ be surface with $H^1(\O_S) = 0$ and $p_g(S)>0$, and let Chern $r$, $c_1$ and $c_2$ be Chern classes such that semistability implies stability. If $r$ is odd, we have
\[\VW_{1^r,c_1,c_2}(S,y) \in \QQ(y) \subset \QQ(\sqrt y)\,.\]
\end{proposition}
\begin{proof}
By Proposition \ref{PropUniv} and Remark \ref{RemExpYII}, the Laurent series $A$, $B$ and $C_{ij}$ have coefficients in $\QQ(y)$.
\end{proof}

\section{Toric computations} \label{SecComp}

In order to determine the coefficients of the series
\[A^{(s+1)}_{\underline{\fN}} \quad \text{for} \quad \underline{\fN}\in \mathcal{N}\coloneqq \left\{\underline{\chi(\O_S)},\underline{K_S^2},\underline{K_S\beta^i}, \underline{\beta^i \beta^j}\right\}_{1\leq i\leq j\leq s}\,,\]
 up to some degree $N$. Cf.\ \cite{Goe}, and as we will explain in this section, it suffices to evaluate the integrals
\begin{equation}\label{EqIntegralRepost}
Q_n(S,\beta) = \int_{[\Hilb^{n}(S)]} \frac{ \prod_{i=1}^s e\left(R\pi_*\O(\beta_i) - \RHom_\pi(\I^{[n_{i-1}]},\I^{[n_i]}(\beta_i))\right)}{e\left(N_{L,0}^{[n]}\right)}
\end{equation}
for $|n|\leq N$ on $\PP^2$ and $\PP^1\times\PP^1$ and sufficiently many different $\beta$ (see Section~\ref{SecTrace} for notation and definitions).

Let $\omega_{2,1^s}$ denote the $\QQ$-vector space of cobordism classes of surfaces with $s$-tuples of line bundles, as defined in \cite{LP}, and let $B$ be a basis. For rank $s+1 = 2$, we could take
\[B = \left([\PP^1\times \PP^1 , \O],  [\PP^2, \O], [\PP^2, \omega_S^{\pm 1}]\right)\,,\]
and for rank $s + 1 = 3$
\[B = \left([\PP^1\times \PP^1 , (\O,\O)],  [\PP^2, (\O,\O)], [\PP^2, (\omega_S^{\pm 1},\O)] , [\PP^2, (\O,\omega_S^{\pm 1})], [\PP^2, (\omega_S,\omega_S)] \right)\,.\]

We can view the symbols $\underline \fN$ as coordinate functions on $\omega_{2,1^s}$. For a surface $S$, with an $s$-tuple $\beta\in (H^2(S,\ZZ))^s$, we write
\[\underline \fN (S,\beta) = \fN\,,\]
so we have
\[
	\underline{\chi(\O_S)}(S,\beta) = \chi(\O_S) \,, \quad
	\underline{K_S^2} (S,\beta) = K_S^2 \,, \quad
	\underline{K_S\beta^i}(S,\beta) = K_S\beta^i \,, \quad \ldots \,.
\]
Then the fact that $B$ is a basis can be expressed by the fact that the matrix
\[
M^{(s+1)}\coloneqq
\Big[ \underline \fN (S,\beta) \Big]_{
[S,\beta] \in B , ~
\underline \fN \in \mathcal{N}}
\]
is invertible. For the bases for rank $2$ and $3$ given above, we have
\[
M^{(2)} =
\begin{pmatrix}
1 & 8 & 0 & 0 \\
1 & 9 & 0 & 0 \\
1 & 9 & 9 & 9 \\
1 & 9 & -9 & 9
\end{pmatrix}
\]
and
\[
M^{(3)} =
\begin{pmatrix}
1 & 8 & 0 & 0 & 0 & 0 & 0 \\
1 & 9 & 0 & 0 & 0 & 0 & 0 \\
1 & 9 & 9 & 9 & 0 & 0 & 0 \\
1 & 9 & -9 & 9 & 0 & 0 & 0 \\
1 & 9 & 0 & 0 & 9 & 9 & 0 \\
1 & 9 & 0 & 0 &  -9 & 9 & 0 \\
1 & 9 & 9 & 9 &  9 & 9 & 9
\end{pmatrix}
\]
respectively.

Recall that by Proposition~\ref{PropUnivUnref}, we have
\[\sum_n Q_n(S,\beta)\, q^n= \prod_{\underline{\fN}}(A^{(s+1)}_{\underline{\mathfrak{N}}})^{\fN}\]
for any surface $S$, and curve classes $\beta = (\beta_1,\ldots, \beta_s)$. Taking the natural logarithm, we obtain
\[\log \sum_n Q_n(S,\beta)\, q^n= \sum_{\underline{\fN}} \fN \log A^{(s+1)}_{\underline{\mathfrak{N}}}\,.\]
By definition of $M$, we have
\[
\bigg[
\log \sum_n Q_n(S,\beta) q^n
\bigg]_{[S,\beta] \in B}
=
M \cdot
\bigg[
\log A^{(s+1)}_{\underline{\mathfrak{N}}}
\bigg]_{\underline \fN \in \mathcal{N} }
\,.
\]

Now assume we want to compute the power series $A_{\underline \fN}^{(s+1)}$ up to order $N$. Since $M$ is invertible, it suffices evaluate the integrals $Q_n(S,\beta)$ for all $n\in (\ZZ_{\geq0})^{s+1}$ with $|n| \leq N$. Note that by Proposition \ref{PropUniv}, the discussion above also applies to the refined case.

Let $S$ be any toric surface with a torus $T$, and assume that we have equipped all line bundles appearing in the integral \eqref{EqIntegralRepost} with an equivariant structure. Then, by applying the Atiyah-Bott localization formula, we obtain
\begin{align}\label{EqIntComp}
Q_n(S,\beta)
	& =  \sum_{F\in (\Hilb^n(S))^T} \int \frac{\prod_{i=1}^s e\left(\left(\left.R\pi_*\O(\beta_i) - \RHom_\pi(\I^{[n_{i-1}]},\I^{[n_i]}(\beta_i))\right)\right |_F\right)}{e\left(\left.N_{L,0}^{[n]}\right |_F\right)e(T_{\Hilb^n(S),F})} \notag \\
	& = \sum_{F\in (\Hilb^n(S))^T} \int e\left(\left.- T_{L,0}^{[n]}\right|_F\right)
\end{align}
in which $e()$ denotes the equivariant Euler class for the torus $T\times \CC^*$.
\begin{remark}
In the factor
\[e\left(\left(\left.R\pi_*\O(\beta_i) - \RHom_\pi(\I^{[n_{i-1}]},\I^{[n_i]}(\beta_i))\right)\right |_F\right)\]
in the formula above, the Euler class $e(\,)$ should a priori be the $T$-equivariant Chern class $c^T_{n_{i-1} + n_i}(\,)$, but by \cite[Lemma 6]{CO} and Lemma \ref{LemLocRHom} below, the $K$-theory class
\[\left(\left.R\pi_*\O(\beta_i) - \RHom_\pi(\I^{[n_{i-1}]},\I^{[n_i]}(\beta_i))\right)\right |_F \in K_0^T(F)\]
can be represented by an honest $n_{i-1}+n_i$-dimensional representation of the torus $T$. It follows that the $T$-equivariant top Chern class agrees with $T$-equivariant Euler class.
\end{remark}
\begin{remark}
The compact form of the expression \ref{EqIntComp} is due to the fact that it is obtained by applying the Atiyah-Bott localization formula twice. The virtual version of the formula, due to Graber and Pandharipande \cite{GP}, expresses by definition the contributions \eqref{EqContVW} of nested Hilbert schemes $S_\beta^{[n]}$ to the monopole branch of the Vafa-Witten invariant for a surface $S$ with $p_g(S)>0$. The second time, however, we applied the formula to  Hilbert schemes of points on a toric surface.
\end{remark}

Similarly, we have
\begin{align*}
Q_n(S,\beta,y)
	& =
	\sum_{F\in (\Hilb^n(S))^T} \int
		\frac{\ch\left((K_{L,0}^{[n]})^\frac{1}{2}|_F\right)}{\ch\left(\Lambda^\bullet(N_{L,0}^{[n]\vee}|_F)\right)}\Td\left((T_{L,0}^{[n]})^{\CC^*}|_F\right) \\
	& \quad \times	
	 \frac{\prod_{i=1}^s e\left(\left(\left.R\pi_*\O(\beta_i) - \RHom_\pi(\I^{[n_{i-1}]},\I^{[n_i]}(\beta_i))\right)\right |_F\right)}{e(T_{\Hilb^n(S),F})} \\
	& = \sum_{F\in (\Hilb^n(S))^T} \int  \frac{\ch\left((K_{L,0}^{[n]})^\frac{1}{2}|_F\right)}{\ch\left(\Lambda^\bullet(T_{L,0}^{[n]\vee}|_F)\right)} \,,
\end{align*}
where $\ch$ and $\Td$ denote the $T\times\CC^*$-equivariant Chern character and Todd class respectively. We have suppressed that by convention, we denote the Chern character of $\t$ by $y = \ch(\t)$. Finally, note that the second equation follows from the identity
\[\ch(\Lambda^\bullet(L^*)) = 1- \exp (- \alpha) = \frac{e(L)}{\Td(L)}\]
for any 1-dimensional $T$-representation $L$ with $c_1(L) = \alpha$.

Let $F \in \Hilb^n(S)$ be a $T$-fixed point. Let $0\leq i,j \leq s$, and write
\[I = \I^{[n_i]}_F \quad \text{and} \quad J = \I^{[n_j]}_F\,.\]
The class $\left. T_{L,0}^{[n]}\right|_F$
is a linear combination of classes of the form
\[\Big(R\pi_*M - \RHom_\pi(\I^{[n_{i}]},\I^{[n_j]}\otimes M)\Big)\Big |_F = \chi(M) - R\Hom_S(I,J\otimes M)\,,\]
where $M$ is a $T\times\CC^*$-equivariant line bundle on $S$.
\begin{lemma}\label{LemLocRHom}
Let $\{U_\sigma\}_{\sigma=1,\ldots,e(S)}$ be the maximal open cover of $S$ by affine $T$-fixed subsets, cf.\ \cite[Section 4]{GK17}. Then we have
\begin{equation}\label{EqRhom}
\chi(M) - R\Hom_S(I,J\otimes M) = \sum_{\sigma = 1}^{e(S)} \Gamma(U_\sigma, M) - R\Hom_{U_\sigma}(I|_{U_\sigma},J|_{U_\sigma}\otimes M|_{U_\sigma})\,.
\end{equation}
\end{lemma}
\begin{proof}
Write $U_{\sigma\tau} = U_\sigma\cap U_\tau$ for $\sigma < \tau$. Since $I$ and $J$ are ideal sheaves of $\CC^*$-fixed $0$-dimensional subschemes of $S$, and $U_{\sigma\tau}$ does not contain any fixed points, we have
\begin{align*}
\Gamma(U_{\sigma\tau}, \mathcal{E}xt^i(I,J\otimes M))
	& = \Gamma(U_\sigma \cap U_\tau, \mathcal{E}xt^i(I|_{U_{\sigma\tau}},(J\otimes M)|_{U_{\sigma\tau}}))\\
	& = \Gamma(U_{\sigma\tau}, \mathcal{E}xt^i(\O,M))
\end{align*}
for any $i$ , and a similarly for intersections $U_\sigma \cap U_\tau \cap U_\upsilon$.
Now use the local-to-global spectral sequence and the \v{C}ech complex for the covering $\{U_\alpha\}$ (cf.\ \cite[Section 4.6]{MNOP}), to compare the classes $\chi(M)$ and $R\Hom_S(I,J\otimes M)$.
\end{proof}

Now \cite[Lemma 6]{CO}, and also the proof of \cite[Proposion 4.1]{GK17}, give an explicit expression for the right-hand side of \eqref{EqRhom}. This allows us to compute the integrals $Q_n(S,\beta)$ and $Q_n(S,\beta,y)$. We have implemented the computation in Sage \cite{Sage} for $S= \PP^1\times \PP^1$ and $S = \PP^2$ and for any $\beta$, $n$ and rank. Part of the results for rank $3$ are listed in Appendix \ref{AppPow}.

\section{The universal series $A$}\label{SecProofC}
Fix K3 surface $S$ and non-negative integers $n=(n_0,\ldots,n_s)$. Consider the inclusion
\[ i \colon S^{[n]} \hookrightarrow \Hilb^n(S) = S^{[n_0]} \times \cdots \times S^{[n_s]}\]
of the nested Hilbert scheme. In the case that $n = (k, \ldots, k)$ for some non-negative integer $k$, we write
\[\Delta_{S^{[k]}\times \cdots \times S^{[k]}} \cong S^{[k]} \subset S^{[k]}\times \cdots \times S^{[k]}\]
for the diagonal.

\begin{lemma} \label{LemVanK3}
\[ i_*[S^{[n]}]^\vir =
\begin{cases}
[\Delta_{S^{[k]}\times \cdots \times S^{[k]}}]
	& \text{if } n = (k,\ldots,k)\\
0
	& \text{else}
\end{cases}
\]
\end{lemma}
\begin{proof}
For the first case, see \cite[Theorem 2]{GSY} (note that in this case we have
\[S^{[n]} = S^{[k,\ldots,k]} \cong \Delta_{S^{[k]}\times \cdots \times S^{[k]}}\cong S^{[k]}\,,\]
and the perfect obstruction theory is just the cotangent bundle). For the second case, 
note that we have by Theorem~\ref{TheComp} and Serre duality:
\begin{align*}
i_*[S^{[n_0,\ldots,n_s]}]^\vir
	& = \prod_i e( \chi(\O_S) - \RHom_\pi(\I^{[n_{i-1}]}, \I^{[n_i]})) \cap [\Hilb^n(S)]\\
 	& = (-1)^{n_0 + n_{s}} \prod_i e( \chi(\O_S) - \RHom_\pi(\I^{[n_{i}]}, \I^{[n_{i-1}]})) \cap [\Hilb^n(S)] \\
	& = (-1)^{n_0 + n_{s}} j_*[S^{[n_s,\ldots,n_0]}]^\vir \,,
\end{align*}
where $j$ is the inclusion
\[j\colon S^{[n_s,\ldots,n_0]} \hookrightarrow \Hilb^n(S)\,.\]
Now note that $S^{[n_0,\dots,n_s]}$ or $S^{[n_s,\dots,n_0]}$ is empty, unless $n_0 = \ldots = n_s$.
\end{proof}

Set $\beta = (0,\ldots,0)$, so we have $\beta^i = K_S - \beta_i = 0$ for $i=1,\ldots,s$. By Proposition~\ref{PropUnivUnref} we have
\[\sum_n Q_n(S,\beta) \, q^n = \left(A^{(s+1)}_{\underline{\chi(\O_S)}}\right)^2\,.\]
Recall that $Q_n(S,\beta)$ is by definition integral over the virtual class of $S_\beta^{[n]} = S^{[n]}$. Hence, by Lemma \ref{LemVanK3}, we have $Q_n(S,\beta) = 0$ or $n_0 = \ldots = n_s$. Assume that we have $n = (k,\ldots,k)$ for a non-negative integer $k$. We will compute $Q_n(S,\beta)$.

We let $L = (\O_S,\ldots , \O_S)$ the $(s+1)$-tuple of copies of the trivial line on $S$, so we have $\beta = \beta(L)$. We write
\begin{align*}
E^{[n]} \coloneqq E_L^{[n]} = \pr_0^* \I^{[k]} \otimes \t^0 \oplus \ldots \oplus \pr_s^* \I^{[k]} \otimes \t^{-s}\,
\end{align*}
for the sheaf on $\Hilb^n(S)\times S$, where $\pr_i$ denotes the $i$th projection
\[\pr_i\colon \Hilb^n(S) = S^{[k]} \times \cdots \times S^{[k]} \rightarrow S^{[k]}\,,\]
and its base change to $S$.
We will write $\Delta\colon S^{[k]} \rightarrow S^{[k]} \times \cdots \times S^{[k]}$ for the diagonal embedding (which of course can be identified with the embedding $i$), and denote its base change to $S$ by the same symbol. We have an isomorphism
\[ \Delta^* E^{[n]} \cong \I^{[k]} \otimes \t^0 \oplus \ldots \oplus \I^{[k]} \otimes \t^{-s}\]
of sheaves on $S^{[k]} \times S$. Write
\[E = (\t^0\oplus \ldots\oplus \t^{-s})\otimes \O_S\]
for the vector bundle on $S$. Using the notation Section~\ref{SecTrace}, we have the following equality in the $\CC^*$-equivariant K-group of $S^{[k]}$:
\begin{align}
\Delta^*T_{L,0}^{[n]}
\notag			& = \RHom_\pi(\Delta^*E^{[n]},\Delta^*E^{[n]} \otimes \t)_0 - \RHom_\pi(\Delta^*E^{[n]},\Delta^*E^{[n]})_0 \\
\notag			& \quad - \left(R\Hom(E,E \otimes \t)_0 - R\Hom(E,E)_0 \right) \\
\notag			& = \sum_{i=1}^{s+1} \RHom_\pi(\I^{[k]},\I^{[k]})_0 \otimes \t^i +  \sum_{i=0}^{s} - \RHom_\pi(\I^{[k]},\I^{[k]})_0 \otimes \t^{-i} \\
\label{EqNoCont}	& = T_{S^{[k]}} + \sum_{i=1}^s \left( T_{S^{[k]}} \otimes \t^{-i} - T_{S^{[k]}}\otimes \t^i\right) - T_{S^{[k]}} \otimes \t^{s+1}\,.
\end{align}
Note that $T_{S^{[k]}}$ has even rank, so we have
\[e(T_{S^{[k]}} \otimes \t^{-i}) = e(T_{S^{[k]}} \otimes \t^{i})\,,\]
and hence
\begin{align*}
e\left(\sum_{i=1}^s \left( T_{S^{[k]}} \otimes \t^{-i} - T_{S^{[k]}}\otimes \t^i\right)\right)
	& = \prod_{i=1}^s  \frac{e\left(T_{S^{[k]}} \otimes \t^{-i} \right)} {e\left(T_{S^{[k]}}\otimes \t^i\right)} \\
	& = 1
\end{align*}
It follows that we have
\begin{align*}
Q_n(S,\beta)
	& = \int_{[S_\beta^{[n]}]^\vir} \frac{1}{e(N_{L,0}^{[n]})} \\
	& = \int_{S^{[k]}}  \frac{1}{e\left(\Delta^*\left(T_{L,0}^{[n]} \right)^\mathrm{mov} \right)} \\
	& = \int_{S^{[k]}} e(T_{S^{[k]}} \otimes \t^{s+1}) \\
	& = \int_{S^{[k]}} e(T_{S^{[k]}}) \\
	& = e(S^{[k]})\,.
\end{align*}
By G\"ottsche's formula \cite{Go90}, we have
\begin{align*}
\left(A^{(s+1)}_{\underline{\chi(\O_S)}} \right)^2
	& = \sum_{n\in(\ZZ_{\geq0})^{s+1}} Q_n(S,\beta) \, q^{|n|} \\
	& = \sum_{k\in \ZZ_{\geq0}} e(S^{[k]}) \, q^{k(s+1)} \\
	& = \left(\prod_{k \geq 1} \frac{1}{1-q^{k(s+1)}}\right)^{24}\,.
\end{align*}
It follows that
\[ F_0^{(s+1)} A^{(s+1)}_{\underline{\chi(\O_S)}} (-1)^{s+1} = \frac{1}{s+1}  \left(\prod_{k \geq 1} \frac{1}{1-q^{k(s+1)}}\right)^{12}\,,\]
which is the specialization of the right-hand side of \eqref{EqK3GenSer} at $y=1$, as expected.

Now consider the integral
\[
Q_n(S,\beta,y) = \Bigg[\int_{[S_\beta^{[n]}]^\vir} \frac{\ch\left((K_{L,0}^{[n]})^{\frac{1}{2}}\right)}{\ch(\Lambda^\bullet((N_{L,0}^{[n]})^\vee))}\Td\left((T_{L,0}^{[n]})^{\CC^*}\right)\Bigg]_{\ch(\t) = y}\,.
\]
Since $T_{S^{[k]}}$ is self-dual and has even rank, we have by \cite[equation 2.28]{T} the following equation in the $\CC^*$-equivariant $K$-group on $S^{[k]}$:
\begin{align*}
\frac{\left(\det(T_{S^{[k]}} \otimes \t^{-i} - T_{S^{[k]}}\otimes \t^i)^\vee\right)^\frac{1}{2}}{\Lambda^\bullet ((T_{S^{[k]}} \otimes \t^{-i} - T_{S^{[k]}}\otimes \t^i)^\vee)}
	& = \frac{\det(T_{S^{[k]}}^* \otimes \t^{i}) \otimes \Lambda^\bullet(T_{S^{[k]}}^*\otimes \t^{-i})}{\Lambda^\bullet (T_{S^{[k]}}^* \otimes \t^{i})} \\
	& = 1\,.
\end{align*}
It follows that, as above, the middle terms of \eqref{EqNoCont} do not contribute to $Q_n(S,\beta,y)$.
\begin{remark}
Note that taking the square root involves a choice. First of all, our choice here is consistent with the one we made before. More importantly, after choosing a root $\sqrt \t$, the choice is unique up to $2$-torsion in $\Pic(S^{[n]}_\beta)$, which is killed by $\ch$.
\end{remark}

Since $T_{S^{[k]}}$ is self-dual, so is $K_{S^{[k]}}$, and hence $\ch(K_{S^{[k]}}) = 1$. Writing $r=s+1$, we find
\begin{align*}
Q_n(S,\beta,y)
	& = \Bigg[\int_{S^{[k]}} \frac{\ch\left(\left(\det(T_{S^{[k]}} - T_{S^{[k]}} \otimes \t^{r})^\vee\right)^\frac{1}{2}\right)}{\ch(\Lambda^\bullet ( - T_{S^{[k]}} \otimes \t^{r})^\vee)}\Td(T_{S^{[k]}})\Bigg]_{\ch(\t) = y} \\
	& = \Bigg[\int_{S^{[k]}} \ch\left(K_{S^{[k]}} \otimes \t^{rk}\right) \ch(\Lambda^\bullet(T_{S^{[k]}} \otimes \t^{-r})) \Td(T_{S^{[k]}})\Bigg]_{\ch(\t) = y} \\
	& = y^{rk} \sum_{i=0}^{2k} (-1)^i y^{-ri}  \int_{S^{[k]}} \ch( \Lambda^i T_{S^{[k]}})  \Td(T_{S^{[k]}})  \\
	& = y^{rk} \sum_{i=0}^{2k} (-1)^i y^{-ri}  \chi( \Lambda^i T_{S^{[k]}})  \\
	& \eqqcolon \chi_{-y^{r}}(S^{[k]})\,.
\end{align*}
The generating series of $\chi_y$-genera of the Hilbert schemes $S^{[k]}$ has been computed in \cite{GS}. It follows that we have
\begin{align*}
\left(A^{(r)}_{\underline{\chi(\O_S)}}(y)\right)^2
	& = \sum_{n\in(\ZZ_{\geq0})^{r}} Q_n(S,\beta,y) \, q^{|n|} \\
	& = \sum_{k\in\ZZ_{\geq 0}} \chi_{-y^{r}}(S^{[k]}) q^{rk}\\
	& = \prod_{k\geq 1} \frac{1}{(1-q^{rk})^{20}(1-y^{-r}q^{rk})^2 (1 - y^r q^{rk})^2}\,.
\end{align*}
We conclude
\begin{align*}
A^{(r)}(y)
	& = F_0^{(r)}(y) A^{(r)}_{\underline{\chi(\O_S)}}(y) (-1)^{r} \\
	& = \frac{1}{[r]_y}   \prod_{k\geq 1} \frac{1}{(1-q^{rk})^{10}(1-y^{-r}q^{rk}) (1 - y^r q^{rk})} \\
	& = \frac{y^\frac{1}{2} - y^{-\frac{1}{2}}}{\phi_{-2,1}(q^r,y^r)^\frac{1}{2}\, \tilde\Delta(q^r)^\frac{1}{2}} \,,
\end{align*}
proving Theorem~\ref{Theorem4}.

\section{Smooth components}
In the case that the monopole branch of the moduli space of $\CC^*$-fixed Higgs pairs is smooth, there is a direct method to compute the Vafa-Witten invariants. Let $S$ be a surface with $H^1(\O_S) = 0$, $p_g>0$, and assume that $\Pic(S)$ is generated by a smooth very ample canonical curve $C$. In this case, the only Seiberg-Witten basic classes of $S$ are $0$ and $K_S$. For rank $2$, the monopole branch \[\M_{1^2}=\M_{1^2,K_S,c_2}\subset (\N_{2,\omega_S,c_2}^\bot )^{\CC^*}\]
is smooth precisely when $c_2 = 0,1,2,3$ \cite{TT}. In particular, the virtual class is given by the Euler class of the obstruction bundle and the Vafa-Witten invariants can be computed using the intersection theory of (smooth) nested Hilbert schemes of points on the surface and the smooth canonical curve. This method, which is carried out in \cite{TT} (unrefined) and \cite{T} (refined), can be generalized to rank $3$, but only for $c_2 = 0,1,2$. We have done the computation in this setting, and have found that they confirm our results (see the discussion after Theorem \ref{Theorem3}).

Let $(E,\phi)\in \M_{1^3,K_S,c_2}$ be a Higgs pair, so $E$ can be written as
\[E = I_0\otimes \omega_S^a \oplus I_1 \otimes \omega_S^b \oplus I_2 \otimes \omega_S^c\]
where $I_i\in S^{[n_i]}$ for $i = 0,1,2$ and $a,b,c\in \ZZ$ such that $a+b+c = 1$. Moreover, we have $\phi = (\phi_1,\phi_2)$ for non-zero homomorphisms
\begin{align*}
\phi_1\colon I_0\otimes \omega_S^a
	& \rightarrow I_1 \otimes \omega_S^{b+1}\,,\\
\phi_2\colon I_1\otimes \omega_S^b
	& \rightarrow I_2 \otimes \omega_S^{c+1}\,.
\end{align*}

\begin{lemma} \label{LemBasic}
We have $(a,b,c) = (1,0,0)$.
\end{lemma}
\begin{proof}
Slope semistability of $E$ implies that
\[c \leq \frac{1}{3}\,, \quad \frac{b+c}{2} \leq \frac{1}{3}\,.\]
On the other hand, by the existence of the maps $\phi_1$ and $\phi_2$ we have
\[a\leq b+1 \,, \quad b \leq c+1\,.\]
It is easy to see that the only integral solution to these inequalities together with $a+b+c=1$ is $(a,b,c) = (1,0,0)$.
\end{proof}

\begin{proposition}\label{PropSmoothMod}
Let $S$ be given as above. Then $\M_{1^3,K_S,c_2}$ is smooth if and only if $c_2\leq 2$. In particular, we have
\begin{align*}
\M_{1^3,K_S,1}
	& \cong \left(S^{[1]} \times |K_S|\right) \sqcup
			\C\,;\\
\M_{1^3,K_S,2}
	& \cong \left(S^{[2]} \times |K_S|\right) \sqcup
			\left(S^{[1]} \times |K_S|\right) \sqcup
			\left(S^{[1]}\times \C\right) \sqcup
			\C_{|K_S|}^{[2]}
\end{align*}
in which
\[
\begin{tikzcd}
\C \arrow[r]\arrow[d] &  S \\
\vert K_S \vert &
\end{tikzcd}
\]
is the universal canonical curve, and $\C_{|K_S|}^{[2]} \rightarrow |K_S|$ the relative Hilbert scheme of pairs of points. 
\end{proposition}
\begin{proof}
Note that for $I_i \in S^{[n_i]}$, $i=0,1,2$, we have
\[c_2(I_0\otimes \omega_S \oplus I_1 \oplus I_2) = n_0 + n_1 + n_2\,.\]
By Lemma \ref{LemBasic}, we find
\[\M_{1^3,K_S,c_2} \cong \bigsqcup_{|n| = c_2, n_0\geq n_1} S_{(0,K_S)}^{[n_0,n_1,n_2]}\,,\]
where we have used that $S_{(0,K_S)}^{[n_0,n_1,n_2]}$ is empty whenever $n_0 < n_1$. In particular, we have
\begin{align*}
\M_{1^3,K_S,1}
	& = S_{(0,K_S)}^{[1,0,0]} \sqcup S_{(0,K_S)}^{[0,0,1]} \\
	& \cong S^{[1]} \times |K_S|\, \sqcup\, \C\,;\\
\M_{1^3,K_S,2}
	& = S_{(0,K_S)}^{[2,0,0]} \sqcup S_{(0,K_S)}^{[1,1,0]} \sqcup S_{(0,K_S)}^{[1,0,1]} \sqcup S_{(0,K_S)}^{[0,0,2]} \\
	& \cong \left(S^{[2]} \times |K_S|\right) \sqcup
			\left(S^{[1]} \times |K_S|\right) \sqcup
			\left(S^{[1]}\times \C \right) \sqcup
			\C_{|K_S|}^{[2]}\,.
\end{align*}
The total spaces of the universal canonical curve $\C$, and of relative Hilbert scheme of points $\C^{[2]}_{|K_S|}$ are smooth by the the assumption that $K_S$ is very ample.

The component
\[S_{(0,K_S)}^{[1,1,1]} \cong (S\times\C) \cup (\Delta_S \times |K_S|) \xhookrightarrow{i} S\times S \times |K_S|\]
of $\M_{1^3,K_S,3}$ has two irreducible components with non-empty intersection. More generally, let $c_2 \geq 3$. For an ideal sheaf $I$ on $S$, let $Z_I$ denote the corresponding subscheme. Then the component
\begin{align*}
S_{0,K_S}^{[1,1,c_2-2]}
	& \cong S_{K_S}^{[1,c_2-2]} \\
	& = \{p\in S, I\in S^{[c_2-2]}, C \in |K_S| : Z_I \subset C \cup p\}
\end{align*}
has two components given by the conditions $p\in Z_I$, and $Z_I \subset C$ respectively. Hence, it is singular at points in the intersection given by $p\in Z_I\subset C$. It follows that $\M_{1^r,K_S,c_2}$ is singular.
\end{proof}

The connected components of $\M_{1^3,K_S,1}$, together with the restrictions of the universal sheaf on $\M_{1^3,K_S,1}\times S$ are given as follows:
\begin{align*}
S^{[1]} \times |K_S|
	& \,, \quad \I^{[1]}\otimes \omega_S \oplus \t^{-1} \oplus \O_{|K_S|}(1) \otimes \t^{-2} \,;\\
\C
	& \,, \quad \omega_S \oplus \t^{-1} \oplus j^*(\I^{[1]}\otimes \O_{|K_S|}(1))\otimes \t^{-2} \,.
\end{align*}
in which $j\colon \C\times S \rightarrow S^{[1]} \times |K_S| \times S$ is the inclusion. We have suppressed pull-backs along the several projections. For $c_2 = 2$, we have
\begin{align*}
S^{[2]} \times |K_S|
	& \,, \quad
		\Big( \I^{[2]} \otimes \omega_S \Big) \oplus
		\t^{-1} \oplus
		\Big( \O_{|K_S|}(1) \otimes \t^{-2} \Big) \,;\\
S^{[1]} \times |K_S|
	& \,, \quad
		\Big( \I^{[1]} \otimes \omega_S \Big) \oplus
		\Big( \I^{[1]} \otimes \t^{-1} \Big) \oplus
		\Big( \O_{|K_S|}(1) \otimes \t^{-2} \Big) \,;\\
S^{[1]}\times \C\
	& \,, \quad
		\Big(\I^{[1]} \otimes \omega_S \Big) \oplus
		\t^{-1} \oplus
		\Big(j^*\left(\I^{[1]}\otimes \O_{|K_S|}(1)\right)\otimes \t^{-2} \Big)\,;\\
\C_{|K_S|}^{[2]}
	& \,, \quad
		\omega_S \oplus
		\t^{-1} \oplus
		\Big( j_2^*\left(\I^{[2]}\otimes \O_{|K_S|}(1)\right)\otimes \t^{-2} \Big)\,.
\end{align*}
in which we have written
\[j_2\colon \C_{|K_S|}^{[2]}\times S \hookrightarrow S^{[2]} \times |K_S| \times S\]
for the inclusion. Again, we have suppressed pull-backs along projections. Now define Higgs fields $\phi = (\phi_1,\phi_2)$ by the several natural inclusions of ideal sheaves.

As the moduli spaces are smooth, we can compute the virtual class of each component by taking the Euler class of the obstruction bundle. Write $H \coloneqq c_1(\O_{|K_S|}(1))$. Using Theorem~\ref{TheVir}, we find

\begin{align*}
[S^{[1]} \times |K_S|]^\vir
	& = e(K_S^* + \Omega_{|K_S|}(1)) \\
	&= (-1)^{p_g}\cdot[C]\,,\\
[\C]^\vir
	& = e(\Omega_{|K_S|}(1)) \\
	& = (-1)^{p_g-1}\cdot [C]\,,\\
[S^{[2]} \times |K_S|]^\vir
	& = e(\omega_S^{[2]*} + \Omega_{|K_S|}(1)) \\
	& = [\C_{|K_S|}^{[2]}]\cap (-H)^{p_g-1}\\
	& = (-1)^{p_g-1}\cdot [C^{[2]}]\,,\\
[S^{[1]} \times |K_S|]^\vir
	& = e(H^0(\omega_S)) \\
	& = 0\,, \\
[S^{[1]}\times \C]^\vir
	& = e(\omega_S^* + \Omega_{|K_S|}(1)) \\
	& = (-1)^{p_g}\cdot [C\times C]\,,\\
[\C_{|K_S|}^{[2]}]^\vir
	& = e(\Omega_{|K_S|}(1)) \\
	& = (-1)^{p_g-1}\cdot [C^{[2]}]\,.
\end{align*}
It follows that the computation of the contribution of the Vafa-Witten invariant reduces to computations in the intersection rings of $C$ and $C^{[2]}$. Using Grothendieck-Riemann-Roch to compute the Chern classes of the relative Hom complexes, this is a straight forward computation. The details are similar to the computations in \cite{TT} and \cite{T}.

\section{Comparison to the G\"ottsche-Kool conjectures}
For rank $2$, the Laurent series that appear in Theorem \ref{Theorem1}, and are defined in Section \ref{SecProof}, are given by
\begin{align*}
A^{(2)}
	& = \frac{1}{y^{-\frac{1}{2}} + y^\frac{1}{2}} A^{(2)}_{\underline{\chi(\O_S)}}(y)\,,\\
B^{(2)}
	& = A^{(2)}_{\underline{K_S^2}}(y)\,,\\
C^{(2)}_{11}
	& = - q^{-\frac{1}{4}} \frac{1}{y^{-\frac{1}{2}} + y^\frac{1}{2}} A^{(2)}_{\underline{\beta^1\beta^1}}(y) A^{(2)}_{\underline{\beta^1K_S}}(y)\,,
\end{align*}
and for rank $3$ by
\begin{align*}
A^{(3)}
	& = \frac{1}{y^{-1} + 1 + y} A^{(3)}_{\underline{\chi(\O_S)}}(y)\,,\\
B^{(3)}
	& = A^{(3)}_{\underline{K_S^2}}(y)\,,\\
C^{(3)}_{11}
	& = q^{-\frac{1}{3}} \frac{1}{y^{-1} + 1 + y} A^{(3)}_{\underline{\beta^1\beta^1}}(y) A^{(3)}_{\underline{\beta^1K_S}}(y)\,,\\
C^{(3)}_{22}
	& = q^{-\frac{1}{3}} \frac{1}{y^{-1} + 1 + y} A^{(3)}_{\underline{\beta^2\beta^2}}(y) A^{(3)}_{\underline{\beta^2K_S}}(y)\,,\\
C^{(3)}_{12}
	& = q^{-\frac{1}{3}} \frac{(1+y)^2}{1 + y + y^2} A^{(3)}_{\underline{\beta^1\beta^2}}(y)\,.
\end{align*}

In Section \ref{SecComp}, we have discussed a method for computing the terms of the generating series $A^{(r)}_{\underline{\fN}}$ appearing above. In Appendix \ref{AppPow} we have listed the first few terms of the rank $3$ power series. The computations allow us to check the equations of Conjectures \ref{ConjR2} and \ref{ConjR3} term by term, leading to Theorem \ref{Theorem2}. As an example, let us just verify one term of $C_{12}^{(3)}$. We have
\begin{align*}
C^{(3)}_{12}
	={}& q^{-\frac{1}{3}} \frac{(1+y)^2}{1 + y + y^2} A^{(3)}_{\underline{\beta^1\beta^2}}(y) \\
	={}& q^{-\frac{1}{3}} \frac{(1+y)^2}{1 + y + y^2}
		\left( 1 + \frac{y^4 + 6y^3 + 6y^2 + 6y + 1}{(y + 1)^2y} \, q + \ldots \right)\,.
\end{align*}
On the other hand we have
\begin{align*}
W(q^\frac{1}{2},y)
	& = \frac{\Theta_{A_2,(0,0)}(q^\frac{1}{2},y)}{\Theta_{A_2,(1,0)}(q^\frac{1}{2},y)} \\
	& = \frac{1 + (y^2 + 2y + 2y^{-1} + y^{-2}) \,q + \ldots}{(y + 1 + y^{-1}) q^\frac{1}{3} + (y^2 + 1 + y^{-2})\,q^\frac{4}{3} + \ldots} \\
	& = q^{-\frac{1}{3}}\frac{1}{y + 1 + y^{-1}}
	\frac{1 + (y^2 + 2y + 2y^{-1} + y^{-2}) \,q + \ldots}{1 + (y - 1 + y^{-1})\,q + \ldots}\\
	& = q^{-\frac{1}{3}}\left(\frac{1}{y + 1 + y^{-1}} + \frac{y^2 + y + 1 + y^{-1} + y^{-2}}{y + 1 + y^{-1}}\,q + \ldots\right)\,,
\intertext{and hence}
W(q^\frac{1}{2},1)
	&= q^{-\frac{1}{3}}\frac{1}{3} ( 1 + 5\,q + \ldots)\,.
\end{align*}
It follows that
\begin{align*}
& W_+(q^\frac{1}{2},y) W_-(q^\frac{1}{2},y) \\
& = W(q^\frac{1}{2},y) + 3 W(q^\frac{1}{2},1) \\
& = q^{-\frac{1}{3}} \left( 1 + \frac{1}{y + 1 + y^{-1}} +  \left( \frac{y^2 + y + 1 + y^{-1} + y^{-2}}{y + 1 + y^{-1}} + 5\right)\,q + \ldots\right)\\
& = q^{-\frac{1}{3}} \left(\frac{y + 2 + y^{-1}}{y + 1 + y^{-1}} + \left( \frac{y^2 + y + 1 + y^{-1} + y^{-2}+ 5 (y+1+y^{-1}) }{y + 1 + y^{-1} }\right)\,q + \ldots\right)\\
& = q^{-\frac{1}{3}} \frac{(1+y)^2}{1 + y + y^2} \left( 1 + \frac{y^4 + 6y^3 + 6y^2 + 6y + 1}{(y + 1)^2y} \, q + \ldots \right)\,\\
& \equiv C^{(3)}_{12}(y) \mod U^{(3)}_2 \,,
\end{align*}
where have used the notation
\[U^{(3)}_2 = 1 + q^2 \,\QQ(y^\frac{1}{2})[[q]]\subset \QQ(y^\frac{1}{2})(\!(q^\frac{1}{6})\!)^*\,.\]
from the introduction.

\newpage

\appendix
\section{Functions appearing in the G\"ottsche-Kool conjectures} \label{AppFun}
\begin{align*}
\phi_{-2,1}(x,y)
	& \coloneqq (y^\frac{1}{2} - y^{-\frac{1}{2}})^2 \prod_{n=1}^\infty \frac{(1-x^n y)^2(1-x^n y^{-1})^2}{(1-x^n)^4}\\
\tilde\eta(x)
	& \coloneqq \prod_{n\in\ZZ_{>0}} (1-x^n)\\
\tilde\Delta(x)
	& = \prod_{n\in \ZZ_{>0}} (1-x^n)^{24}\\
\theta_2(x,y)
	& \coloneqq \sum_{n\in \ZZ + \frac{1}{2}} x^{n^2}y^n\\
\theta_3(x,y)
	& \coloneqq \sum_{n\in \ZZ} x^{n^2}y^n\\
\Theta_{A_2,(0,0)}(x,y)
	& \coloneqq \sum_{(m,n)\in\ZZ^2} x^{2(m^2 - mn + n^2)} y^{m+n}\\
\Theta_{A_2,(1,0)}(x,y)
	& \coloneqq \sum_{(m,n)\in\ZZ^2} x^{2(m^2 - mn + n^2+m-n+\frac{1}{3})} y^{m+n}\\
W(x,y)
	& \coloneqq \frac{\Theta_{A_2,(0,0)}(x,y)}{\Theta_{A_2,(1,0)}(x,y)}\,.
\end{align*}

The functions $W_\pm(x,y)$ are defined by following polynomial equation in $\omega$:
\[\omega^2 - (W(x,y)^2 + 3W(x,y)W(x,1))\,\omega + W(x,y) + 3W(x,1) = 0 \,.\]
We will use the convention that $W_-(x,y)$ is the one with leading term
\[x^{\frac{2}{3}} (y^{-1} + 1 + y) \,,\]
so we have
\begin{align*}
W_-(q^{\frac{1}{2}},y)
	& = \frac{y^2 + y + 1}{y} \, q^\frac{1}{3} \left( 1 + \frac{2y^2 + 3y + 2}{(y+1)^2} q + \ldots \right) \\
W_+(q^{\frac{1}{2}},y)
	& = \frac{(y+1)^2 \, y}{y^2+y+1} q^{-\frac{2}{3}} \left(1 + \frac{y^4 + 4 y^3 + 3 y^2 + 4 y + 1}{(y+1)^2} q + \ldots \right) \,.
\end{align*}

\newpage

\section{Rank 3 results} \label{AppPow}
We set $q\coloneqq q_0 = q_1 = q_2$ and print the first few terms of $A_{\underline{\fN}}^{(3)}(y)$ for
\[\underline{\mathfrak{N}}\in \left\{\underline{\chi(\O_S)},\underline{K_S^2}, \underline{\beta^1K_S}, \underline{\beta^2K_S}, \underline{\beta^1 \beta^1}, \underline{\beta^2 \beta^2},  \underline{\beta^1 \beta^2}\right\}\,.\]

\begin{align*}
A_{\underline{\chi(\O_S)}}^{(3)}(y)
	& \equiv 1+\frac{y^6+10y^3+1}{y^3} \, q^3
	\mod q^4\\
A_{\underline{K_S^2}}^{(3)}(y)
	& \equiv 1-{\frac { \left( {y}^{2}+y+1 \right) ^{2}}{ \left( y+1 \right) ^{2}y}} \, q \\
	& -{\frac { \left( 2\,{y}^{4}+7\,{y}^{3}+12\,{y}^{2}+7\,y+2 \right) \left( {y}^{2}+y+1 \right) ^{2}}{{y}^{2} \left( y+1 \right) ^{4}}} \, {q}^{2} \\
	 & +\frac{1}{y^3(y+1)^6} \Big(5y^{12}+39y^{11}+150y^{10}+382y^9+705y^8+1002y^7 \\
	 & +1121y^6+1002y^5+705y^4+382y^3+150y^2+39y+5\Big) \, q^3 \mod{q^4} \\
A_{\underline{\beta^1K_S}}^{(3)}(y)
	& \equiv A_{\underline{\beta^2K_S}}^{(3)}(y)\\
	& \equiv 1 + \frac{1}{2}\frac{(y^2 + y + 1)(y - 1)^2}{(y + 1)^2y}  \, q\\
	& +\frac{1}{8}\frac{(23y^4 + 68y^3 + 142y^2 + 68y + 23)(y^2 + y + 1)^2}{(y + 1)^4y^2} \, q^2 \\
	& -\frac{y^2+y+1}{16 y^3(y+1)^6}\Big(15y^{10}+244y^9+1006y^8+2790y^7+4719y^6+5780y^5 \\
	& +4719y^4+2790y^3+1006y^2+244y+15\Big) \, q^3
	\mod{q^4}\\
A_{\underline{\beta^1 \beta^1}}^{(3)}(y)
	& \equiv A_{\underline{\beta^2 \beta^2}}^{(3)}(y) \\
	& \equiv 1 -\frac{1}{2}\frac{y^4 + 3y^3 + 6y^2 + 3y + 1}{(y + 1)^2y} \, q \\
	& -\frac{1}{8}\frac{1}{(y + 1)^4y^2}
	\Big(5y^8 + 30y^7 + 109y^6 + 218y^5 + 280y^4 \\
	& + 218y^3 + 109y^2 + 30y + 5\Big) \, q^2 \\
	& + \frac{1}{16 y^3(y+1)^6}\Big(11y^{12}+115y^{11}+571y^{10}+1868y^9+4205y^8+6845y^7 \\
	& + 8026y^6+6845y^5+4205y^4+1868y^3+571y^2+115y+11\Big) \, q^3
	\mod{q^4}\\
A_{\underline{\beta^1 \beta^2}}^{(3)}(y)
	& \equiv 1 + \frac{y^4 + 6y^3 + 6y^2 + 6y + 1}{(y + 1)^2y} \, q \\
	& -\frac{y^6 + y^5 + 8y^4 + 8y^3 + 8y^2 + y + 1}{(y + 1)^2y^2} \, q^2 \\
	& + \frac{y^6+3y^4+4y^3+3y^2+1}{y^2 (y+1)^2} \, q^3
	\mod{q^4}
\end{align*}

\bibliographystyle{habbrv}

\bibliography{CiVW}{}

\begin{thebibliography}{10}

\bibitem{CO}
E.~Carlsson and A.~Okounkov.
\newblock Exts and vertex operators.
\newblock {\em Duke Math. J.}, 161(9):1797--1815, 2012.

\bibitem{EGL}
G.~Ellingsrud, L.~G{\"o}ttsche, and M.~Lehn.
\newblock On the cobordism class of the {H}ilbert scheme of a surface.
\newblock {\em J. Algebraic Geom.}, 10(1):81--100, 2001.

\bibitem{FG}
B.~Fantechi and L.~G{\"o}ttsche.
\newblock Riemann-{R}och theorems and elliptic genus for virtually smooth
  schemes.
\newblock {\em Geom. Topol.}, 14(1):83--115, 2010.

\bibitem{GSY2}
A.~Gholampour, A.~Sheshmani, and S.-T. Yau.
\newblock Localized {D}onaldson-{T}homas theory of surfaces, 2017,
  arXiv:1701.08902.

\bibitem{GSY}
A.~Gholampour, A.~Sheshmani, and S.-T. Yau.
\newblock Nested {H}ilbert schemes on surfaces: {V}irtual fundamental class.
\newblock {\em Journal of {D}ifferential {G}eometry (to appear).}, 2018,
  arXiv:1701.08899.

\bibitem{GT}
A.~Gholampour and R.~P. Thomas.
\newblock Degeneracy loci, virtual cycles and nested {H}ilbert schemes, 2017,
  arXiv:1709.06105.

\bibitem{GT2}
A.~Gholampour and R.~P. Thomas.
\newblock Degeneracy loci, virtual cycles and nested {H}ilbert schemes {II},
  2018, preprint.

\bibitem{Go90}
L.~G\"{o}ttsche.
\newblock The {B}etti numbers of the {H}ilbert scheme of points on a smooth
  projective surface.
\newblock {\em Math. Ann.}, 286(1-3):193--207, 1990.

\bibitem{Goe}
L.~G{\"o}ttsche.
\newblock A conjectural generating function for numbers of curves on surfaces.
\newblock {\em Comm. Math. Phys.}, 196(3):523--533, 1998.

\bibitem{GK17}
L.~G{\"o}ttsche and M.~Kool.
\newblock Virtual refinements of the {V}afa-{W}itten formula, 2017,
  arXiv:1703.07196.

\bibitem{GK}
L.~G{\"o}ttsche and M.~Kool.
\newblock Refined $\mathrm{SU}(3)$ {V}afa-{W}itten invariants and modularity,
  2018, arXiv:1808.03245.

\bibitem{GNY}
L.~G\"{o}ttsche, H.~Nakajima, and K.~Yoshioka.
\newblock Instanton counting and {D}onaldson invariants.
\newblock {\em J. Differential Geom.}, 80(3):343--390, 2008.

\bibitem{GS}
L.~G\"{o}ttsche and W.~Soergel.
\newblock Perverse sheaves and the cohomology of {H}ilbert schemes of smooth
  algebraic surfaces.
\newblock {\em Math. Ann.}, 296(2):235--245, 1993.

\bibitem{GP}
T.~Graber and R.~Pandharipande.
\newblock Localization of virtual classes.
\newblock {\em Invent. Math.}, 135(2):487--518, 1999.

\bibitem{LP}
Y.-P. Lee and R.~Pandharipande.
\newblock Algebraic cobordism of bundles on varieties.
\newblock {\em J. Eur. Math. Soc. (JEMS)}, 14(4):1081--1101, 2012.

\bibitem{MNOP}
D.~Maulik, N.~Nekrasov, A.~Okounkov, and R.~Pandharipande.
\newblock Gromov-{W}itten theory and {D}onaldson-{T}homas theory. {I}.
\newblock {\em Compos. Math.}, 142(5):1263--1285, 2006.

\bibitem{MT}
D.~Maulik and R.~P. Thomas, 2018, in preparation.

\bibitem{Mo}
T.~Mochizuki.
\newblock {\em Donaldson type invariants for algebraic surfaces}, volume 1972
  of {\em Lecture Notes in Mathematics}.
\newblock Springer-Verlag, Berlin, 2009.
\newblock Transition of moduli stacks.

\bibitem{Sage}
T.~{Sage Developers}.
\newblock {\em {S}ageMath, the {S}age {M}athematics {S}oftware {S}ystem
  ({V}ersion 6.10.beta7)}, 2015.
\newblock {\tt http://www.sagemath.org}.

\bibitem{TT}
Y.~Tanaka and R.~P. Thomas.
\newblock Vafa-{W}itten invariants for projective surfaces {I}: stable case.
\newblock {\em Jour. Alg. Geom (to appear)}, 2017, arXiv:1702.08487.

\bibitem{T}
R.~P. Thomas.
\newblock Equivariant {K}-theory and refined {V}afa-{W}itten invariants, 2018,
  arXiv:1810.00078.

\bibitem{VW}
C.~Vafa and E.~Witten.
\newblock A strong coupling test of {$S$}-duality.
\newblock {\em Nuclear Phys. B}, 431(1-2):3--77, 1994.

\end{thebibliography}

\end{document}